\documentclass[a4paper,reqno,11pt]{amsart}

\usepackage[margin=3.5cm]{geometry}
\usepackage[utf8]{inputenc}
\usepackage[T1]{fontenc}

\usepackage{amsmath, amsfonts, amssymb, amsthm, amscd, bbm}
\usepackage{graphicx}
\usepackage{psfrag}
\usepackage{perpage}
\usepackage{url}
\usepackage{mathtools}
\usepackage{color}
\usepackage{mathrsfs}
\usepackage{mathabx}
\usepackage{scalerel}
\usepackage{enumerate}
\usepackage{comment}
\usepackage{enumitem}
\usepackage[normalem]{ulem}
\usepackage{caption,subcaption}
\usepackage{subdepth}

\usepackage{dsfont} 

\usepackage{microtype}
\usepackage{hyperref}

\usepackage{hyperref}
\allowdisplaybreaks[0]

\numberwithin{equation}{section}

\newtheorem{theorem}{Theorem}[section]
\newtheorem{lemma}[theorem]{Lemma}
\newtheorem{proposition}[theorem]{Proposition}
\newtheorem{corollary}[theorem]{Corollary}

\newtheorem{conjecture}[theorem]{Conjecture}

\theoremstyle{definition}

\newtheorem{remark}[theorem]{Remark}

\theoremstyle{remark}

\newcommand{\E}{\mathbb{E}}
\newcommand{\N}{\mathbb{N}}
\renewcommand{\P}{\mathbb{P}}

\newcommand{\R}{\mathbb{R}}

\newcommand{\T}{\mathbb{T}}

\def\bs{\boldsymbol}

\newcommand\bP{\ensuremath{\bs{\mathrm{P}}}}
\newcommand\bE{\ensuremath{\bs{\mathrm{E}}}}

\newcommand{\cC}{{\ensuremath{\mathcal C}} }

\newcommand{\cT}{{\ensuremath{\mathcal T}} }

\newcommand{\weight}{\textup{\textrm{w}}}

\renewcommand{\theta}{\vartheta}
\renewcommand{\rho}{\varrho}

\newcommand{\effR}[3]{{\rm R}_{\rm eff}^{#1}(#2 \leftrightarrow #3)} 

\newcommand{\limn}{\lim\limits_{n \rightarrow \infty}}

\newcommand{\ErdosRenyi}{Erd\H{o}s--R\'enyi }
\newcommand{\diam}{\operatorname{diam}}
\newcommand{\Exc}{\operatorname{Exc}}
\newcommand{\height}{\operatorname{ht}}
\newcommand{\BP}{\operatorname{BP}}
\newcommand{\TV}{\operatorname{TV}}

\MakePerPage[2]{footnote} 

\date{\today}

\title{Repeat times and a two-weight UST model}

\author[U. Ambroggio]{Umberto De Ambroggio}
\address{Department of Mathematics\\
National University of Singapore\\
10 Lower Kent Ridge Road, 119076 Singapore
}
\email{umbidea@gmail.com}

\thanks{Email: \href{mailto:umbidea@gmail.com}{\texttt{umbidea@gmail.com}}, \href{mailto:luca.makowiec@uni-leipzig.de}{\texttt{luca.makowiec@uni-leipzig.de}}} 

\author[L. Makowiec]{Luca Makowiec}
\address{University of Leipzig\\ 
Institute of Mathematics\\
Augustusplatz 10, 04109 Leipzig, Germany
}
\email{luca.makowiec@uni-leipzig.de}

\keywords{disordered system, uniform spanning tree, random graphs}
\subjclass[2010]{Primary: 60K35;  Secondary: 82B41, 82B44, 05C05}

\begin{document}

\begin{abstract}
    We study a model of random weighted uniform spanning trees on the complete graph with $n$ vertices, where each edge is assigned a weight of $n^{1+\gamma}$ with probability $1/n$ and $1$ otherwise. Whenever $\gamma$ is large enough, we prove that the diameter of the resulting tree is typically of order $n^{1/3} \log n$, up to a $\log \log n$ correction. Our approach uses estimates on repeat times for selecting components in a critical Erdős–Rényi graph, as well as concentration bounds on the sums of diameters of these components.
\end{abstract}

\maketitle



\section{Introduction}
Let $G = K_n$ be the complete graph on $n$ vertices and, for $\gamma \in \R$, assign random weights
\begin{equation} \label{eq:weight_dist}
    \weight(e) \coloneqq  
    \begin{cases}
        n^{1+\gamma} & \text{with probability } \frac{1}{n}, \\
        1 & \text{otherwise},
    \end{cases}
\end{equation}
to the edges $e$ of $G$. Denote by $\mathbb{T} = \mathbb{T}(G)$ the set of spanning trees (i.e., connected cycle-free subgraphs) of $G$. Given a realization of the weights $(\weight(e))_{e \in E(G)}$, we define the weighted \textit{uniform spanning tree} (UST) measure on $\mathbb{T}(G)$ as
\begin{equation}\label{eq:PomegaT}
    \bP^\weight_G(\cT = T) 
    	:= \frac{1}{Z^{\weight}} \prod_{e \in T} \weight(e)
\end{equation}
with normalization constant
\begin{equation}\label{eq:Zomega}
    Z^{\weight} :=  \sum_{T \in \T} \prod_{e \in T} \weight(e).
\end{equation}
When the weights and the underlying graph are clear, we simply write $\bP_\cT(\cdot)$ for this law. We remark that $\bP_\cT(\cdot)$ is itself a random variable as it is a function of the (random) weights $(\weight(e))_{e \in E(G)}$.

\smallskip

This model is a specific instance of a \textit{random spanning tree in random environment} (RSTRE) as studied in \cite{MSS23, MSS24, K24, Mak24}. For instance, the papers \cite{MSS24, K24} consider the complete graph with weight distribution
\begin{equation*}
    \weight(e) = \exp( - \beta \omega_e),
\end{equation*}
where $(\omega_e)_{e \in E(G)}$ are i.i.d. uniform random variables on $[0,1]$, and study the effect of the parameter $\beta = \beta(n) \geq 0$ on the diameter of typical weighted USTs. Namely, for $\beta \leq n^{1 - o(1)}$ the typical diameter is of order $\sqrt{n}$, the same order as the diameter of the unweighted (i.e., all weights are equal) UST, whereas for $\beta \geq n^{4/3 + o(1)}$ the diameter is of order $n^{1/3}$, the same order of that of the random \textit{minimum spanning tree} (MST). It is conjectured, see \cite[Conjecture 1.3]{MSS24}, that there is some intermediate regime for the choice of parameter $\beta$ such that the diameter smoothly interpolates between the two power laws $\sqrt{n}$ and $n^{1/3}$ corresponding to the UST and MST, respectively.

Furthermore, in \cite{K24} and \cite{Mak24}, the authors study ``local'' properties of the RSTRE depending on the choice of $\beta$. Notably, Theorem~1.3 of \cite{Mak24} shows that there is a sharp transition of the local limit of the RSTRE whenever $\beta = n^{\alpha}$ and $\alpha$ crosses the critical threshold $\alpha_c = 1$: it either agrees with the UST ($\alpha < 1$) or the MST ($\alpha > 1$) local limit. We refer to \cite[Chapter 4]{LP16} for an introduction to USTs, and to the introduction of \cite{ABGM17} for a clear overview of the historical development and main results concerning the MST. See also the introduction of \cite{MSS24} for more background regarding the RSTRE and its relation to the UST and MST. We remark here that our choice of weights can be reparameterized to $\weight(e) = \exp(-\beta \omega_e)$ where $\omega_e$ has the law of minus a Bernoulli random variable with parameter $1/n$, and $\beta = (1+\gamma) \log n$.

\subsection{Main result}

\smallskip

In a similar spirit to the results of \cite{MSS23, MSS24, K24}, we study the diameter of the random spanning tree $\cT$, i.e.,\ the maximum length of a path in $\cT$ connecting vertices of $G$. For now, we only focus on the case when $\gamma$ is large. Our main result is the following. 

\begin{theorem} \label{T:diam_extreme}
    Let $\gamma \geq 5$. For any $\varepsilon > 0$ there exists $A = A(\varepsilon)$ and $n_0 = n_0(\varepsilon)$ such that for $n \geq n_0$
    \begin{equation}
         \widehat{\P} \big( A^{-1} n^{1/3} \log n \leq \diam(\cT) \leq A n^{1/3} \log n \log \log n \big) \geq 1 - \varepsilon,
    \end{equation}
    where $\widehat{\P}(\cdot)$ is the averaged law $\E[\bP_\cT(\cdot)]$.
\end{theorem}

\noindent We remark that the upper bound in Theorem~\ref{T:diam_extreme} contains an additional $\log \log n$ correction term compared to the lower bound. We suspect that this term arises from technical limitations of our approach, and that the typical diameter of the tree should be of order $n^{1/3} \log n$.  
\begin{remark}
    As shown in \cite{ABR09}, the typical diameter of a random minimum spanning tree is of order $n^{1/3}$. Hence, in the regime of large $\gamma$, the model sees the appearance 
    of an additional $\log n$ correction term beyond what is observed in the MST and in the RSTRE of \cite{MSS24, K24} whenever $\beta$ is large.
\end{remark}

The techniques developed in \cite{MSS23, MSS24} are not applicable to our disorder distribution, and hence new proof ideas are required. In particular, in Theorem~\ref{T:repeat_times}, we characterize repeat-times of sampling (with replacement) components in a critical \ErdosRenyi random graph in terms of an observable of a drifted Brownian motion, as described in \cite{Ald97}. This result will be needed in the proof of our main result (Theorem~\ref{T:diam_extreme}), but it may be of independent interest on its own. Furthermore, our proof uses a concentration argument for the sum of diameters (or typical distances) in randomly chosen components. As far as we are aware, neither one of these results has been studied in the literature before. In contrast, when $\gamma < 0$, the methods of \cite[Theorem~2.3]{MSS23} (see also \cite[Theorem~1.1]{MNS21}) can be adapted, and in Section~\ref{S:small_gamma} we show that in this case the diameter does behave as in the UST case.


\subsection{Critical \ErdosRenyi random graphs and branching processes} \label{SS:ER-graphs}
Due to the structure of the weights in \eqref{eq:weight_dist}, there is a close connection between this model and critical \ErdosRenyi random graphs, for which we briefly recall some notation. Before doing so, however, we need to recall some basic graph-theoretic terminology.

Given a simple and undirected graph $G=(V,E)$ and vertices $u,v\in V$, we write $u\leftrightarrow v$ if there is a path from $u$ to $v$ (or equivalently, from $v$ to $u$ as we are considering undirected graphs). If such a path exists, then we write $d_G(u,v)$ for the graph distance, that is, the length of a shortest path connecting $u$ and $v$. The diameter of $G$ is the maximal distance over all pairs of distinct vertices in the graph. The component of a vertex $u\in V$ is the set of vertices $\mathcal{C}(u)=\mathcal{C}(u;G)\coloneqq \{v\in V: v\leftrightarrow u\}$.
We write $|\mathcal{C}(u)|$ for the number of vertices in $\mathcal{C}(u)$. We denote by $\mathcal{C}_i=\mathcal{C}_{i}(G)$ the $i$-th largest component in $G$ (if two components have the same size, we list them in such a way that the one containing the vertex of smallest label comes first in the ordering) and write $|\mathcal{C}_i|$ for its size. In particular, $|\mathcal{C}_1|=\max_{u\in V}|\mathcal{C}(u)|$. 

\smallskip

We now recall the definition and some basic properties of \ErdosRenyi random graphs.
Let $G = K_n$ be the complete graph with vertices labeled by $\{1, \ldots, n\} =: [n]$. Given $p \in [0,1]$, for each edge $ e \in E(G)$ we independently keep the edge with probability $p$, and remove it otherwise (i.e., with probability $1-p$). We denote the resulting (random) graph by $\mathbb{G}(n,p)$, and write $\P_{n,p}$ and $\E_{n,p}$ for the corresponding law and expectation, respectively. When the parameters are clear, we often suppress the subscripts and write $\P,\E$ instead. 

It is well-known that, by letting $p=p(n)=\mu/n$ with $\mu>0$, a phase transition occurs when $\mu$ passes one. More precisely, when $\mu<1$ (fixed), then w.h.p.\ there is no component in $\mathbb{G}(n,p)$ containing more than $O(\log(n))$ vertices; when $\mu>1$ (fixed), then w.h.p.\ there is a unique giant component of order $\Theta(n)$ and all other components are of size $O(\log(n))$. Finally, when $\mu=1$ (the so-called \textit{critical} case), then the size of a largest component is of order $\Theta(n^{2/3})$. In fact, if 
\begin{equation*}
    p = \frac{1 + \lambda n^{-1/3}}{n}, \qquad \lambda \in \R,
\end{equation*}
then largest components are still of size $\Theta(n^{2/3})$; this is the so-called \textit{critical window}. We refer the reader to e.g.\ \cite{vdH17} for a detailed introduction to the topic and proofs of the above statements.

We will often use a coupling between the random weights in \eqref{eq:weight_dist} and the $\mathbb{G}(n,p)$ model by keeping an edge $e$ if and only if $\weight(e) = n^{1+\gamma}$. In particular, for any $u \in V = [n]$ we have
\begin{equation*}
    \cC(u) = \big\{ v \in V : \exists \text{ a path between $u$ and $v$ of edges $e$ with } \weight(e) = n^{1+\gamma} \big\}.
\end{equation*}
Whenever we refer to the components of the random graph in the context of the random weights, we implicitly assume this coupling.

\smallskip

Finally, we will also need the concept of \textit{branching process}, which we briefly recall now. We refer the reader to e.g.\ \cite{vdH17} for an introduction to this model. Let $X$ be a random variable taking values in the set of non-negative integers $\mathbb{N}_0$ and let $(X_{i,k}:i,k\in \mathbb{N})$ be a family of i.i.d.\ random variables distributed as $X$.
A (discrete-time) branching process is a stochastic process $(Z_k:k\in \mathbb{N}_0)$ with values in $\mathbb{N}_0$ which is constructed recursively as follows. We set $Z_0\coloneqq 1$ and iteratively define
\begin{equation*}
    Z_k\coloneqq \sum_{i=1}^{Z_{k-1}}X_{i,k}, \text{ for }k\in \mathbb{N}.
\end{equation*}

We can interpret $Z_k$ as the number of alive individuals at time $k$ in a population started with one individual, where each member of the population, independently from the other individuals, produces $X$ offspring in the next generation before dying. From this point of view, $X_{i,k}$ is the number of children (in generation $k$) of the $i$-th individual belonging to the $(k-1)$-th generation. There is a natural (random) rooted tree $T_{\BP}$, with root $\rho$, associated to such a process. For each $k \in \N$, the random variable $Z_k$ then corresponds to the number of vertices at distance $k$ from the root. The height of $T_{\BP}$, denoted by $\height(T_{\BP})$, is the maximal distance of a vertex to the root, or equivalently the maximum $k$ such that $Z_k>0$.

\subsection{Outline and proof ideas } \label{SS:proof_ideas}
In Section~\ref{S:RG_estimates}, we collect several facts and estimates about connected components in critical random graphs, whereas in Section~\ref{S:size_biased_GNP} we consider the problem of sampling (with replacement) components in a critical random graph and study how many components have to be sampled until a component is sampled twice for the first time. The proof of Theorem~\ref{T:diam_extreme} is split into two parts: in Section~\ref{S:lower_bound}, we cover the proof of the lower bound in Theorem~\ref{T:diam_extreme}, whereas Section~\ref{S:upper_bound} covers the upper bound. Finally, in Section~\ref{S:small_gamma} we discuss the case $\gamma < 5$, and prove that for $\gamma < 0$ the diameter of the spanning tree is of order $\sqrt{n}$. 

For the proof of the lower bound in Theorem~\ref{T:diam_extreme}, with the Aldous-Broder algorithm in mind, we run a random walk started at a randomly chosen vertex and keep track of all the connected components that the walk visits. The results from Section~\ref{S:size_biased_GNP} will imply that approximately $n^{1/3}$ components are visited before a repeated component is seen by the random walk. We then construct a path in the tree consisting of the union of these $n^{1/3}$ components. Roughly speaking, each component $\cC_i$ will contribute at least $c |\sqrt{\cC_i}|$ to the length of the path. The lower bound then follows from a concentration inequality applied to the sum of $n^{1/3}$ terms of the form $c \sqrt{|\cC_i|}$.  See also Section~\ref{SS:lower_idea} for a more technical explanation.


The rough idea of the proof for the upper bound is the following. The length of a path in $\cT$ between two vertices $u$ and $v$ can be bounded by decomposing the path into two paths connecting the vertices $u$ and $v$ to a ``large'' set $\mathcal{L}$, e.g.\ $\mathcal{L} = \cC_1$, and a path inside $\mathcal{L}$. When $\mathcal{L} = \cC_1$, the latter only contributes an additive factor of order $n^{1/3}$ to the overall length. As $\mathcal{L}$ has a size of order $n^{2/3}$, the path, say from $u$ to $\mathcal{L}$, typically visits about $n^{1/3}$ many components before reaching $\cC_1$. We will show that the sum of the diameters of these $n^{1/3}$ components concentrates well enough. However, to guarantee that the paths from all vertices (i.e., not only typical) to $\mathcal{L}$ visit on the order of at most $n^{1/3}$ many components, we need to enlarge $\mathcal{L}$ to include more than just the largest component. This will be achieved by setting $\mathcal{L}$ to be the union of the largest $(\log n)^3$ many components so that the size of $\mathcal{L}$ is at least $n^{1/3} \log n$. A union bound over random walks started in the $(\log n)^3$ largest components incurs our $\log \log n$ correction term. 


\subsection*{General notation.}
Here we collect some standard notation used throughout the article. We write $\mathbb{N}=\{1,2,\dots\}$ for the set of positive integer and $\mathbb{N}_0=\{0\}\cup \mathbb{N}$. For $n\in \mathbb{N}$ we set $[n]\coloneqq \{1,\dots,n\}$. When talking about sequences of random variables, we use the abbreviation i.i.d.\ to mean that they are independent and identically distributed. Moreover, we write $\text{Bin}(N,q)$ for the binomial distribution with parameters $N\in \mathbb{N}$ and $q\in [0,1]$; we write $\text{Ber}(q)$ for the Bernoulli distribution with parameter $q\in [0,1]$. Sometimes we write $=_d$ to represent equality in distribution. Given functions $f,g:\mathbb{N}\mapsto [0,\infty)$, we write either $f\ll g$ or $f=o(g)$ when $f(n)/g(n)\rightarrow 0$ as $n\rightarrow \infty$, and we write $f\gg g$ when $g(n)/f(n)\rightarrow 0$  as $n\rightarrow \infty$. We write $f=O(g)$ when there is a constant $C\geq 0$ such that $f(n)\leq Cg(n)$ for all large enough $n$. We write $f=\Theta(g)$ when $f=O(g)$ and $g=O(f)$.


\section{Random graph estimates} \label{S:RG_estimates}
In this section, we collect some properties about components in a critical binomial random graph, which will be used in the proof of Theorem~\ref{T:diam_extreme}. We remark that some of the results listed below are already known in the literature, however, for completeness and clarity’s sake, we include some of the proofs.

\subsection{Largest components}

Recall that we write $\mathcal{C}(i)$ for the component of vertex $i\in [n]$, whereas $\cC_1$ is the largest of such components. As already mentioned in Section~\ref{SS:ER-graphs}, the largest component of a binomial random graph considered in the critical window contains approximately $n^{2/3}$ vertices. The following theorem, which corresponds to Theorem~5.1 in \cite{vdH17}, gives tail bounds on the size of $\cC_1$. Stronger tail bounds can be found e.g.\ in \cite{DeARob1} and \cite{DeA25}.
\begin{theorem} \label{AP_T:C_1_size}
    Suppose that $p= (1 + \lambda n^{-1/3})/n$ for $\lambda \in \R$ fixed. Then, for any $\varepsilon > 0$, there exist constants $A = A(\lambda, \varepsilon)$ and $n_0 = n_0(\lambda, \varepsilon)$ such that
    \begin{equation}
        \P_{n,p} \Big( \frac{1}{A} n^{2/3} \leq |\cC_1| \leq A n^{2/3} \Big) \geq 1 - \varepsilon,
    \end{equation}
    for all $n \geq n_0$.
\end{theorem}
Later on we will need bounds not only on the size of the largest component, but also on the size of the union of the $k$ largest components, for some $k=k(n)$. The next result says that, if $k$ is at least of poly-logarithmic order (in the number of vertices), then the number of vertices contained in the first $k$ largest components is at least $n^{2/3}\log(n)$. 
\begin{lemma} \label{L:union_of_largest}
    Fix $\varepsilon>0$. There is $n_0=n_0(\varepsilon)\in \mathbb{N}$ such that, if $k\coloneqq \lceil (\log n)^3 \rceil$ and $n\geq n_0$, then
    \begin{equation}
        \mathbb{P}_{n,1/n}\Big(\sum_{i=1}^k |\mathcal{C}_i| \geq n^{2/3} \log n\Big)\geq 1-\varepsilon.
    \end{equation}
\end{lemma}
\begin{proof}
    Note that, since $\sum_{i=1}^k |\mathcal{C}_i|\geq k|\mathcal{C}_k|$ (and the $|\mathcal{C}_i|$ are non-negative random variables), if $\sum_{i=1}^k |\mathcal{C}_i|<n^{2/3}\log(n)$ then the events
    \begin{equation*}
         \Big\{ k|\mathcal{C}_k|<n^{2/3}\log(n) \Big\} \ \text{ and } \ \Big\{ \sum_{i=1}^{k-1} |\mathcal{C}_i|<n^{2/3}\log(n) \Big\}
    \end{equation*}
    both occur. Therefore, we can write
    \begin{align*}
        \mathbb{P}\Big(\sum_{i=1}^k |\mathcal{C}_i|<n^{2/3}\log(n)\Big)&\leq \mathbb{P}\Big(\sum_{i=1}^{k-1} |\mathcal{C}_i|<n^{2/3}\log(n), |\mathcal{C}_k|<n^{2/3}\log(n)/k\Big)\\
        &\leq \mathbb{P}\Big(|\mathcal{C}_k|<n^{2/3}\log(n)/k\mid \sum_{i=1}^{k-1} |\mathcal{C}_i|<\lceil n^{2/3}\log(n)\rceil\Big).
    \end{align*}
    On the event $\{ \sum_{i=1}^{k-1} |\mathcal{C}_i|<\lceil n^{2/3}\log(n)\rceil\}$, we see that $|\mathcal{C}_k|$ stochastically dominates the size of a largest component in the binomial random graph $\mathbb{G}(n-\lceil n^{2/3}\log(n)\rceil,1/n)$. Set $m\coloneqq n-\lceil n^{2/3}\log(n)\rceil$ and note that
    \begin{equation*}\label{}
        \frac{1}{n} = \frac{1}{m} \frac{n - \lceil n^{2/3}\log(n)\rceil}{n}  \geq \frac{1 - 2 m^{-1/3} \log m}{m}\eqqcolon \frac{1-\gamma(m)}{m}.
    \end{equation*}
    Thus, writing $|\mathcal{C}_1^m|$ for the size of a largest component in the random graph $\mathbb{G}(m,(1-\gamma(m))/m)$, we obtain
    \[\mathbb{P} \Big( |\mathcal{C}_k|<n^{2/3}\log(n)/k\mid \sum_{i=1}^{k-1} |\mathcal{C}_i|<\lceil n^{2/3}\log(n)\rceil \Big) \leq \mathbb{P} \big( |\mathcal{C}_1^m|<n^{2/3}\log(n)/k \big).\]
    Moreover, for all large enough $n$ 
    \begin{equation}\label{edgeprob}
        \frac{n^{2/3}\log(n)}{k}\leq 2\gamma(m)^{-2}\log(m\gamma(m)^3)\frac{C}{\log\log(m)}
    \end{equation}
    for some constant $C>0$. However, from e.g.\ \cite{NP07} it is known that
    \[\frac{|\mathcal{C}_1^m|}{2\gamma(m)^{-2}\log(m\gamma(m)^3)}\longrightarrow 1\]
    in probability as $m\rightarrow \infty$. Thus, using (\ref{edgeprob}), we obtain
    \[\mathbb{P} \big(|\mathcal{C}_1^m|<n^{2/3}\log(n)/k \big)\leq \mathbb{P}\Big(\frac{|\mathcal{C}_1^m|}{2\gamma(m)^{-2}\log(m\gamma(m)^3)}<\frac{C}{\log\log(m)}\Big)\leq \varepsilon \]
    for all large enough $m$, completing the proof.
\end{proof}

When $p$ is in the critical window, then the maximum diameter is typically achieved by one of the largest few components, and these components have a diameter of order at most $n^{1/3}$. In fact, the number of cycles in these components is bounded in probability and locally they are tree-like. 
We will make use of the following theorem, and remark that stronger tail bounds are proven in \cite{NP08}.
\begin{theorem}[Theorem~1.3 in \cite{NP08}] \label{T:Nachmias_diam}
    Fix $\varepsilon > 0$ and $\lambda \in \R$. There exists $A = A(\varepsilon, \lambda) < \infty$ and $n_0=n_0(\varepsilon,\lambda)\in \mathbb{N}$ such that
\begin{equation}
    \P_{n,p} \big( \exists \text{connected component } \mathcal{C} \text{ with } \diam(\mathcal{C}) > An^{1/3} \big) < \varepsilon
\end{equation}
for all $n\geq n_0$.
\end{theorem}

\subsection{Typical components}

Next, we give bounds on the sizes of a typical component, that is, the component of some fixed vertex $v$ or, by symmetry, of the vertex $1$.

\begin{lemma} \label{L:RG_typical}
    Let $p=(1+\lambda n^{-1/3})/n$ where $\lambda\in \mathbb{R}$ is fixed, and let $r>0$. There exist constants $C=C(\lambda)>0$ and $n_0=n_0(\lambda,r)\in \mathbb{N}$ such that, if $k\leq rn^{2/3}$, then
    \begin{equation}
        \P( \cC(v) \geq k) \leq C(n^{-1/3}+k^{-1/2})
    \end{equation}
    for all $n\geq n_0$.
\end{lemma}
We refer to Proposition~5.2 of \cite{vdH17} for a proof. In the next lemma, we provide some further estimates about typical components; the ranges of $j$ are not optimal, but sufficient for our purposes.
\begin{lemma} \label{L:collection_RG_facts}
    Consider the critical \ErdosRenyi random graph $\mathbb{G}(n,p)$ with $p=p(n)=1/n$. The following facts hold true.
    \begin{enumerate}[label=(\roman*)]
        \item \label{enu:collection_1} There exists a constant $c > 0$ such that, if $1\leq j \leq  n^{1/5}$, then 
        \begin{equation}
            \mathbb{P} \big( |\mathcal{C}(1)|=j \big) \geq \frac{c}{j^{3/2}}.
        \end{equation}
        
        \item \label{enu:collection_2} Suppose that $1\leq j\leq n^{1/5}$. Then
        \begin{equation}
            \mathbb{P} \big( \mathcal{C}(1) \text{ contains a cycle} \ \big| \ |\mathcal{C}(1)|=j \big )\ll 1 .
        \end{equation}
    \end{enumerate}
\end{lemma}
The proof of the above lemma makes use of an \textit{exploration process}, which is a standard algorithmic procedure to reveal the components of a (simple, undirected) graph. This procedure also allows us to establish both statements in the lemma. We refer the reader to e.g.\ \cite{NP07} or \cite{DeARob1} and references therein for other instances where an exploration process has been used to study component sizes in random graphs.

Fix an ordering of the $n$ vertices with $v$ listed first. At every step of the algorithm, each vertex is in one of three possible statuses: \textit{active}, \textit{unseen} or \textit{explored}. Let us denote by $\mathcal{A}_t$,  $\mathcal{U}_t$ and $\mathcal{E}_t$ the (random) sets of active, unseen and explored vertices at the end of step $t\in \mathbb{N}_0$, respectively. Then, for any given $t\in \mathbb{N}_0$, we can partition the vertex set as $[n]=\mathcal{A}_t\cup \mathcal{U}_t\cup \mathcal{E}_t$ (a disjoint union), so that in particular the set of unseen vertices satisfies $\mathcal{U}_t=[n]\setminus (\mathcal{A}_t\cup \mathcal{E}_t)$ at each step $t$.\\ 

\textbf{Algorithm 1}. At time $t=0$, vertex $v$ is declared active whereas all other vertices are declared unseen, so that $\mathcal{A}_0=\{v\}$, $\mathcal{U}_0=[n]\setminus\{v\}$ and $\mathcal{E}_0=\emptyset$. For every $t\in \mathbb{N}$, we first pick $u_t$ according to the following rule:
\begin{itemize}
	\item [(a)] If $|\mathcal{A}_{t-1}|\geq 1$, we let $u_t$ be the first active vertex (here and in what follows, the term \textit{first} refers to the ordering that we have fixed at the beginning of the procedure).
	\item [(b)] If $|\mathcal{A}_{t-1}|=0$ and $|\mathcal{U}_{t-1}|\geq 1$, we let $u_t$ be the first unseen vertex.
	\item [(c)] If $|\mathcal{A}_{t-1}|=0=|\mathcal{U}_{t-1}|$ (so that $\mathcal{E}_{t-1}=[n]$), we instead halt the procedure.
\end{itemize}
Now, denote by $\mathcal{D}_t$ the set of \textit{unseen} neighbors of $u_t$, i.e., we set
\[\mathcal{D}_t\coloneqq \left\{x\in \mathcal{U}_{t-1}\setminus\{u_t\}:u_t\sim x \right\}.\]
Subsequently we update
\[\mathcal{U}_t\coloneqq \mathcal{U}_{t-1}\setminus (\mathcal{D}_t\cup \{u_t\}),\text{ }\mathcal{A}_t\coloneqq (\mathcal{A}_{t-1}\setminus\{u_t\})\cup \mathcal{D}_t \text{ and }\mathcal{E}_t\coloneqq \mathcal{E}_{t-1}\cup \{u_t\}.\]
\begin{remark} 
	Note that, since in the procedure \textbf{Algorithm 1} we explore \textit{one vertex} at each step, we have $\mathcal{A}_t\cup \mathcal{U}_t\neq \emptyset$ for every $t\leq n-1$ and $\mathcal{A}_n\cup \mathcal{U}_n=\emptyset$ (as $\mathcal{E}_n=[n]$). Thus the algorithm runs for $n$ steps.
\end{remark}
We run the above algorithm on (a realization of) $\mathbb{G}(n,1/n)$. 
Let $\eta_t$ be the (random) number of unseen vertices that we add to the set of active nodes at time $t$ in \textbf{Algorithm 1}. Since at the end of each step $i$ in which $|\mathcal{A}_{i-1}|\geq 1$ we remove the (active) vertex $u_i$ from $\mathcal{A}_{i-1}$ (after having revealed its unseen neighbors), we have the recursion
\begin{itemize}
	\item $|\mathcal{A}_t|=|\mathcal{A}_{t-1}|+\eta_t-1$, if $|\mathcal{A}_{t-1}|>0$;
	\item $|\mathcal{A}_{t}|=\eta_t$, if $|\mathcal{A}_{t-1}|=0$.
\end{itemize}
Observe that, when $|\mathcal{A}_{t-1}|>0$, then, if we denote $\mathcal{F}_{k}$ the $\sigma-$algebra generated by $\eta_1,\dots,\eta_k$ (with $\mathcal{F}_0$ being the trivial $\sigma-$algebra), we have
\begin{equation}\label{laweta}
    (\eta_t\mid \mathcal{F}_{t-1})=_d\text{Bin}(n-(t-1)-|\mathcal{A}_{t-1}|,1/n).
\end{equation}
Moreover, 
\[|\mathcal{C}(v)|=\min\{t\in [n]:|\mathcal{A}_{t}|=0\}.\]
\begin{proof}[Proof of Lemma~\ref{L:collection_RG_facts}] 
We start by establishing \ref{enu:collection_1}. Recall that we are interested in the case where $p=1/n$. Without loss of generality, we can assume that $j \geq j_0$ for some constant $j_0$ as otherwise the desired inequality can be obtained by making $c=c(j_0)$ sufficiently small.

Thanks to the recursive nature of the number of active vertices, we can write 
\begin{equation}\label{first}
    \mathbb{P}(|\mathcal{C}(1)|=j)=\mathbb{P}\Big(1+\sum_{i=1}^t(\eta_i-1)>0\text{ }\forall t\in [j-1], 1+\sum_{i=1}^j(\eta_i-1)=0\Big).
\end{equation}
To lower bound the probability on the right-hand side of (\ref{first}) we perform the following steps. 
\begin{enumerate}
    \item Firstly, we replace the $\eta_i$ with (simpler) i.i.d.\ random variables $\xi_i$ having the $\text{Bin}(n,1/n)$ distribution, thus obtaining a mean-zero, finite variance random walk. This operation of replacing the $\eta_i$ is carried out at a cost $\ll j^{-3/2}$.
    \item Subsequently, we replace the $\xi_i$ with i.i.d.\ random variables $P_i$ having the $\text{Poi}(1)$ distribution. This operation is also carried out at a cost $\ll j^{-3/2}$.
    \item Lastly, we use a \textit{ballot theorem} to bound from below the probability that the (mean-zero) random walk with i.i.d.\ increments $P_i-1$ stays above zero for $j-1$ steps and hits zero at time $j$. This last probability is of order $j^{-3/2}$, whence thanks to (1) and (2) above we can conclude that $j^{-3/2}$ is also a lower bound for the probability of interest.
\end{enumerate}

We begin by carrying out the steps (1) and (2). In particular, we show that
    \begin{multline*}
        \mathbb{P}\Big(1+\sum_{i=1}^t(\eta_i-1)>0\text{ }\forall t\in [j-1], 1+\sum_{i=1}^j(\eta_i-1)=0\Big)\\
        \geq \mathbb{P}\Big(1+\sum_{i=1}^t(\xi_i-1)>0\text{ }\forall t\in [j-1], 1+\sum_{i=1}^j(\xi_i-1)=0\Big)-o(j^{-3/2}),
    \end{multline*}
    where the $\xi_i$ are i.i.d.\ with the $\text{Bin}(n,1/n)$ distribution. Then we conclude by carrying out step (3).
    
    In what follows we will need a uniform upper bound on the number of active vertices that we can have at any step of the exploration process. 
    Arguing as in Lemma~4.2 of \cite{DeARob1}, and denoting by $\tau_0$ the first time $t$ at which $|\mathcal{A}_t|=0$, for any $r\in (0,1)$ we have that
    \begin{align*}
        \mathbb{P}(\exists t\in [\tau_0\wedge j]:|\mathcal{A}_t|>x)&\leq \mathbb{P} \big( \exists t\in [ j]:1+\sum_{i=1}^t(\xi_i-1)>x \big)\\
        &\leq 2e^{-rx-rj}\mathbb{E} \big[ \exp(r\text{Bin}(nj,1/n) \big];
    \end{align*}
    The last inequality follows from the classical Doob's sub-martingale inequality, which states that for a non-negative sub-martingale $(X_k:k\in \mathbb{N}_0)$ and for all $N\in \mathbb{N},x>0$, we have
    \[\mathbb{P}(\exists k\in [N]:X_k>x)\leq \mathbb{E}[X_N]/x.\]
    Using that $e^r-1\leq r+r^2$ for $r\in [0,1]$, we see that $\mathbb{E}[\exp(r\text{Bin}(nj,1/n)]\leq e^{rj+r^2j}$
    and minimizing with respect to $r$ we obtain 
    \begin{equation}\label{upactive}
        \mathbb{P}(\exists t\in [\tau_0\wedge j]:|\mathcal{A}_t|>x)\leq 2e^{-x^2/(4j)}.
    \end{equation}
    Let $x\geq 1$ be a positive integer. Define the (good) event $\mathcal{G}\coloneqq \{|\mathcal{A}_t|\leq x\text{ }\forall t\in [\tau_0\wedge j]\}$. Recalling (\ref{laweta}), we let
    \[\eta^+_t\coloneqq \sum_{i=1}^{t-1+|\mathcal{A}_{t-1}|}J_{i,t} \quad \text{ for each }t\in [j],\]
    where $(J_{i,t}:i,t\in \mathbb{N})$ is a (doubly) infinite sequence of i.i.d.\ random variables with the $\text{Ber}(1/n)$ law, also independent of all other random quantities involved. Define $\xi_t\coloneqq \eta_t+\eta^+_t$ for each $t\in [j]$. Note that the collection of $\xi_t$ are i.i.d.\ with $\text{Bin}(n,1/n)$ law. Using that
    \[\Big\{1+\sum_{i=1}^t(\eta_i-1)>0\text{ }\forall t\in [j-1], 1+\sum_{i=1}^j(\eta_i-1)=0\Big\}=\{\tau_0=j\}\]
    and setting $\mathcal{H}_m\coloneqq \{\eta^+_i=0 \text{ }\forall i\leq m\}$,
    we write
    \begin{align*}
        \mathbb{P}&\Big(1+\sum_{i=1}^t(\eta_i-1)>0\text{ }\forall t\in [j-1], 1+\sum_{i=1}^j(\eta_i-1)=0\Big)\\
        &\geq \mathbb{P}\Big(1+\sum_{i=1}^t(\eta_i-1)>0\text{ }\forall t\in [j-1], 1+\sum_{i=1}^j(\eta_i-1)=0,\, \mathcal{G},\, \mathcal{H}_j\Big)\\
        &=\mathbb{P}\Big(1+\sum_{i=1}^t(\xi_i-1)>0\text{ }\forall t\in [j-1], 1+\sum_{i=1}^j(\xi_i-1)=0, \, \mathcal{G},\, \mathcal{H}_{\tau_0\wedge j}\Big)\\
        &\geq \mathbb{P}\Big(1+\sum_{i=1}^t(\xi_i-1)>0\text{ }\forall t\in [j-1], 1+\sum_{i=1}^j(\xi_i-1)=0, \, \mathcal{G} \Big)-\mathbb{P}(\sum_{i=1}^{\tau_0\wedge j}\eta^+_i\geq 1, \, \mathcal{G}).
    \end{align*}

    Let
    \begin{equation*}
         x=  \Big\lceil 8  \sqrt{j\log(j)} \, \Big\rceil = O(j).
    \end{equation*}
   By Markov's inequality we obtain
    \begin{align*}
        \mathbb{P} \Big( \sum_{i=1}^{\tau_0\wedge j}\eta^+_i\geq 1,\mathcal{G} \Big)\leq \mathbb{P}\Big(\sum_{i=1}^{\tau_0\wedge j}\sum_{h=1}^{i-1+x}J_{h,i}\geq 1\Big)&\leq \mathbb{P}\Big(\sum_{i=1}^{ j}\sum_{h=1}^{i+x}J_{h,i}\geq 1\Big)\\
        &=O\Big(\frac{jx}{n}\vee \frac{j^2}{n}\Big)= O\Big(\frac{j^2}{n} \Big).
    \end{align*} 

    Setting $R_t\coloneqq 1+\sum_{i=1}^t(\xi_i-1)$, we arrive at
    \begin{align*}
        \mathbb{P} \Big( 1+\sum_{i=1}^t(\eta_i-1)&>0\text{ }\forall t\in [j-1], 1+\sum_{i=1}^j(\eta_i-1)=0 \Big)\\
        &\geq \mathbb{P}(R_t>0\text{ }\forall t\in [j-1], R_j=0)-\mathbb{P}(\mathcal{G}^c)-O \big(\frac{j^2}{n} \big)\\
        &\geq \mathbb{P}(R_t>0\text{ }\forall t\in [j-1], R_j=0)-2e^{-x^2/(4j)}- O\big(\frac{j^2}{n} \big),
    \end{align*}
    where for the last inequality we have used (\ref{upactive}). Our choice of $x$ guarantees that
    \begin{equation*}
        2e^{-x^2/(4j)} \leq  \frac{2}{j^2}.
    \end{equation*}
    Finally, by Theorem 2.10 of \cite{vdH17}, we can couple the $\xi_i$ with i.i.d.\ random variables $P_i$ with the $\text{Poi}(1)$ law such that $\mathbb{P}(\xi_i\neq P_i)\leq 1/n$ for each $i\in [j]$. A union bound then yields
    \[\mathbb{P}(\exists i\in [j]:\xi_i\neq P_i)\leq \frac{j}{n}= O\big(\frac{j^2}{n} \big) .\]
    Hence, setting $S_t\coloneqq 1+\sum_{i=1}^t(P_i-1)$ for $t\in [j]$, we have
    \begin{align*}
        \mathbb{P} \big( R_t>0\text{ }\forall t\in [j-1], R_j=0 \big) &\geq \mathbb{P} \big( R_t>0\text{ }\forall t\in [j-1], R_j=0, \xi_i=P_i\text{ }\forall i\in [j] \big)\\
        &\geq \mathbb{P} \big( S_t>0\text{ }\forall t\in [j-1], S_j=0 \big)-  O\big(\frac{j^2}{n} \big).
    \end{align*}
    By Theorem $9$ in \cite{Add_Reed}, we obtain
    \[\mathbb{P}\big( S_t>0\text{ }\forall t\in [j-1], S_j=0 \big) \geq \frac{c}{j^{3/2}}\]
    for some constant $c>0$ which depends solely on the law of $P_1$.
    Combining all the previous estimates and choosing a smaller constant $c$, gives the desired conclusion that
    \[\mathbb{P} \big( |\mathcal{C}(1)|=j \big)\geq \frac{c}{j^{3/2}}\]
    since $j^2/n \vee 2/j^2 \leq c' j^{-3/2}$ for any $c' > 0$ whenever $j_0 \leq j \leq n^{1/5}$ for $j_0$ large enough.
    
    We proceed to the proof of \ref{enu:collection_2}, which stated that
\begin{equation*}
    \mathbb{P}(\mathcal{C}(1) \text{ contains a cycle }\mid |\mathcal{C}(1)|=j)\ll 1.
\end{equation*}
To this end, we notice that, with reference to \textbf{Algorithm 1}, the procedure creates a cycle in $\mathcal{C}(1)$ at some step $k\leq j$ if, and only if, it finds an edge between $u_k$ (the vertex under investigation at step $k$) and one of the active vertices. Denoting by $B_k$ the indicator random variable of the event $\{u_k\sim u \text{ for some }u\in \mathcal{A}_{k-1}\}$ and letting $\tau_0$ be as above, we obtain by item \ref{enu:collection_1}
\begin{align*}
    \mathbb{P} \big( \mathcal{C}(1) &\text{ contains a cycle }\mid |\mathcal{C}(1)|=j \big)\\
    &\leq Cj^{3/2}\mathbb{P} \big( \mathcal{C}(1) \text{ contains a cycle, } |\mathcal{C}(1)|=j \big)\\
    &= Cj^{3/2}\mathbb{P} \big( \mathcal{C}(1) \text{ contains a cycle, }  |\mathcal{C}(1)|=j, \tau_0=j \big)\\
    &= Cj^{3/2}\mathbb{P} \big( \exists k\in [j]:u_k\sim u \text{ for some }u\in \mathcal{A}_{k-1},\tau_0=j \big)\\
    &\leq Cj^{3/2}\mathbb{P} \Big( \sum_{k\in [j]}^{}B_k\geq 1,|\mathcal{A}_{k}|\leq x\text{ }\forall k\in [j-1] \Big) +\mathbb{P}\big( \exists t\in [\tau_0\wedge j]:|\mathcal{A}_t|>x \big)\\
    &\leq Cj^{3/2}\mathbb{P} \Big(\sum_{k\in [j]}^{}B_k \mathbbm{1}_{\{|\mathcal{A}_{k-1}|\leq x\}}\geq 1 \Big)+2e^{-x^2/(4j)}\\
    &=O\big(j^{5/2}xn^{-1}\big)+2e^{-x^2/(4j)},
\end{align*}
where for the third inequality we have used the bound in (\ref{upactive}) whereas the last step is a direct consequence of Markov's inequality and the fact that an edge in $\mathbb{G}(n,1/n)$ is present with probability $1/n$.

By letting
\begin{equation*}
    x = 8 \sqrt{j \log(n) },
\end{equation*}
we obtain that
\begin{equation*}
    O\big(j^{5/2}xn^{-1}\big)+2e^{-x^2/(4j)}=O\big(j^{3} n^{-1} \sqrt{\log(n)}  \big) + O\big(n^{-2}\big),
\end{equation*}
which is $o(1)$ as $j\leq n^{1/5}$. Whence we obtain
\[\mathbb{P}(\mathcal{C}(1) \text{ is cycle free}\mid |\mathcal{C}(1)|=j)=1-o(1),\]
establishing item \ref{enu:collection_2} and completing the proof of the lemma.
\end{proof}

We proceed with a bound on the probability that the components of typical vertices are disjoint.

\begin{lemma} \label{L:same_component}
    There exists a constant $c>0$ such that, for any two distinct $u,v \in [n]$, we have
        \begin{equation}
            \P \big( \mathcal{C}(u) = \mathcal{C}(v) \big) \leq c n^{-1/3}.
        \end{equation}
\end{lemma}
\begin{proof}
Note that for $A\geq 1$ constant we have, by Lemma~\ref{L:RG_typical},
\[\mathbb{P} \big( \mathcal{C}(u) = \mathcal{C}(v) \big) \leq \mathbb{P}\big( \mathcal{C}(u) = \mathcal{C}(v),|\mathcal{C}(u)|\leq An^{2/3} \big)+\frac{C}{A^{1/2}n^{1/3}}.\]
Next, we observe that if $\mathcal{C}(u) = \mathcal{C}(v)$, then there must be a $z\in \mathcal{C}(u)$ such that $v\sim z$. Hence
\[\mathbb{P}(\mathcal{C}(u) = \mathcal{C}(v),|\mathcal{C}(u)|\leq An^{2/3})\leq \mathbb{P}(\exists z\in \mathcal{C}(u):v\sim z,|\mathcal{C}(u)|\leq An^{2/3}).\]
Now 
\begin{align*}
    \mathbb{P}(\exists z\in \mathcal{C}(u):v\sim z,|\mathcal{C}(u)|\leq An^{2/3})&=\mathbb{P}\Big(\sum_{z\in \mathcal{C}(u)}\mathbbm{1}_{\{z\sim v, |\mathcal{C}(u)|\leq An^{2/3}\}}\geq 1\Big)\\
    &\leq \mathbb{E}\Big[\sum_{z\in \mathcal{C}(u)}\mathbbm{1}_{\{z\sim v, |\mathcal{C}(u)|\leq An^{2/3}\}}\Big]\\
    &\leq \sum_{z\leq \lceil An^{2/3} \rceil }\mathbb{P}(z\sim v)=O(An^{-1/3}).
\end{align*}
Thus we conclude that 
\[\P(\mathcal{C}(u) = \mathcal{C}(v) )=O(An^{-1/3}),\]
which is the desired result (since $A\geq 1$ is a constant, independent of $n$).
\end{proof}


\subsection{Total variation distance}

We end this section by giving a bound on the total variation distance between the graph distances in two random graphs with different numbers of vertices. 

We recall that the total variation distance between two random variables $X_1$ and $X_2$ taking values in a countable set $\mathcal{X}$ is defined by
\[ \Vert X_1-X_2 \Vert_{\TV}\coloneqq \sup_{A\subseteq \mathcal{X}} \big| \mathbb{P}(X_1\in A)-\mathbb{P}(X_2\in A) \big|.\]
It is known (see e.g.\ \cite[Theorem~2.9]{vdH17}) that
\begin{equation} \label{eq:TV_couple}
    \Vert X_1-X_2 \Vert_{\TV} = \inf_{X'_i=_dX_i,\text{ }i\in \{1,2\}}\mathbb{P} (X'_1\neq X'_2),
\end{equation}
where the infimum is over all couplings of $X_1$ and $X_2$. As we are assuming that $X_1$ and $X_2$ are discrete, there exists a coupling for which the infimum in \eqref{eq:TV_couple} is attained. Later on we shall make use of the following straightforward lemma.
\begin{lemma} \label{L:TV_fact}
    Suppose that $X,Y \geq 0$ are two random discrete variables with $X \leq M$, for some constant $M$. Then
    \begin{equation*}
        \E[X] \leq \E[Y] + M \Vert X - Y \Vert_{\TV}.
    \end{equation*}
\end{lemma}
\begin{proof}
    Choosing a suitable coupling gives with \eqref{eq:TV_couple} that
    \begin{equation*}
        \E[X] = \E[X \mathbbm{1}_{X = Y}] + \E[X \mathbbm{1}_{X \neq Y}] \leq \E[Y] + M \Vert X - Y \Vert_{\TV}. \qedhere
    \end{equation*}
\end{proof}
We provide one final estimate about distances between pairs of vertices in a typical component when we remove a subset of vertices. Namely, removing less than $n^{1/3}$ many vertices does not drastically change the distribution of said distances.
\begin{lemma} \label{L:d_TV_remove}
    Let $b \in \mathbb{N}$ with $b \leq n-1$, and let $x \in [n-b]$ be a vertex. Denote by $\mathcal{C}_{m,q}(x)$ the component of $x$ in $\mathbb{G}(m,q)$, and let $u_x$ and $v_x$ be two vertices selected uniformly at random from $\mathcal{C}_{m,q}(x)$. Then there exists a universal constant $C > 0$ (independent of $n$ and $b$) such that, for all $n$ large enough,
    \begin{equation}
        \Vert d_{\cC_{n,1/n}(x)}(u_x,v_x) - d_{\cC_{n-b,1/n}(x)}(u_x,v_x)\Vert_{\TV} \leq C b n^{-1/3}.
    \end{equation}
\end{lemma}
\begin{proof}
We can construct $\mathcal{C}_{n,1/n}(x)$ from $\mathcal{C}_{n-b, 1/n}(x)$ by adding $b$ vertices $\{ n-b+1,\dots,n \}$ to $\mathbb{G}(n-b,p)$ and connecting $i\in [n]\setminus [n-b]$ to each $j\in \mathcal{C}_{n-b}(x)$ independently with probability $1/n$. This gives a coupling of the components $\mathcal{C}_{n, 1/n}(x),\mathcal{C}_{n-b,1/n}(x)$ such that $\mathcal{C}_{n-b,1/n}(x) \subseteq \mathcal{C}_{n,1/n}(x)$. If under this coupling all vertices in $[n]\setminus [n-b]$ are not connected to $\mathcal{C}_{n-b,1/n}(x)$, then the components agree so that all the distances (in particular the distance between two uniform random points) agree with one another. Consequently, in view of \eqref{eq:TV_couple} 
    \begin{align*}
        \Vert d_{n,p,\cC(x)}(u_x,v_x) - d_{n-b,p, C(x)}(u_x,v_x)\Vert_{\TV} 
        &\leq \mathbb{P}\Big(\bigcup_{i\in [n]\setminus [n-b]}^{} \bigcup_{j\in \mathcal{C}_{n-b}(x)}^{}\{i\sim j\}\Big)\\
        &\leq b \, \mathbb{P}\Big(\bigcup_{j\in \mathcal{C}_{n-b}(x)}^{}\{n \sim j\}\Big)\\
        &=b \, \mathbb{P}\Big(\sum_{j=1}^{|\mathcal{C}_{n-b}(x)|+1} \mathbbm{1}_{\{n \sim j\}}\geq 1\Big)\\
        &\leq b \,\Big( \mathbb{P}\Big(\sum_{j=1}^{2n^{2/3}} \mathbbm{1}_{\{ n \sim j\}}\geq 1\Big) + O(n^{-1/3}) \Big) \\ 
        &\leq \frac{Cb}{n^{1/3}}
    \end{align*}
    for some $C>0$ (which is independent of $b$ and $n$), where the second-last step used Lemma~\ref{L:RG_typical} (the size of $\cC_{n-b}(x)$ is greater than $n^{2/3}$ with probability at most of order $n^{-1/3}$, and the last step follows from Markov's inequality. 
\end{proof}


\section{Repeat times in random graphs} \label{S:size_biased_GNP}
Conditional on the realization of the \ErdosRenyi random graph $\mathbb{G}(n,1/n)$ with components $\cC_1, \cC_2, \ldots$, denote by $S_1, S_2, \ldots$ a (random) sequence of connected components where, for each $i \geq 1$, we independently pick a component with probability proportional to its size. More precisely, we define a (conditional) probability measure $\bP^n_S$ such that, for any $i \geq 1$ and any components $\cC_{k_1}, \ldots \cC_{k_i}\subset \mathbb{G}(n,1/n)$, we have
\begin{equation} \label{eq:choose_S}
    \bP^n_S \Big( \big(S_1, \ldots, S_i \big)  = \big( \cC_{k_1}, \ldots, \cC_{k_i} \big) \Big) = \prod_{j=1}^{i} \frac{|\cC_{k_j}|}{n}.
\end{equation}
Notice that sampling a component in this fashion is equivalent to choosing a vertex uniformly at random and then taking its component. This observation allows us to write
\begin{equation} \label{eq:E_f(S)}
    \bE_S^n[f(S_i)] = \frac{1}{n} \sum_{x \in V} f \big( \cC(x) \big),
\end{equation}
for any function $f$ on graphs. Furthermore, observe that $\bP^n_S(\cdot)$ is a random variable whose distribution is determined by the randomness of the underlying \ErdosRenyi random graph. 

We denote by 
\begin{equation}
    t_1 \coloneqq  \min \big\{ i \geq 1 : \exists j < i \text{ with } S_i = S_j \big\}  \leq n 
\end{equation}
the first time at which a component appears twice in the sampling procedure described above. The following theorem, which gives information on the distribution of $t_1$, is essentially a combination of the main results from \cite{Ald97} and \cite{CP00}.

\begin{theorem} \label{T:repeat_times}
    Let
    \begin{equation}
        s_n := \frac{1}{n} \sqrt{\sum_{i \geq 1}|\cC_i|^2 }.
    \end{equation}
    Then there exists non-negative random variables $(\theta_i)_{i \geq 1}$ such that, for any $r \geq 0$, 
    \begin{equation*}
       \bP^n_S \big( s_n t_1 >  r \big) \xrightarrow{(d)} e^{-\frac{1}{2} (1 - \sum_{i\geq 1}\theta_i^2 ) r^2} \prod_{i \geq 1}(1 + \theta_i r) e^{- \theta_i r},
    \end{equation*}
    where $\xrightarrow{(d)}$ denotes convergence in distribution.
\end{theorem}
\begin{proof}
    Denote by $(\gamma_i)_{i \geq 1}$ the sequence of excursions, ordered in decreasing length, of a Brownian motion that has drift $-t$ at time $t$ and is reflected at its running minimum. That is, the excursions of the process given by
    \begin{equation*}
        B(t) \coloneqq \big( W(t) - \frac{1}{2}t^2 \big) - \min_{0 \leq s \leq t} \big( W(t) - \frac{1}{2}t^2 \big) \quad \forall t \geq 0,
    \end{equation*}
    where $W(t)$ is a standard Brownian motion.
    
    As established by Aldous \cite{Ald97}, the random sequence $(\cC_1/n^{2/3}, \cC_2/n^{2/3}, \ldots)$, which has zero entries after the last component, converges in distribution to the lengths $(|\gamma_1|, |\gamma_2|, \ldots)$ in the space $\ell^2_{\downarrow}$ consisting of square summable, non-increasing sequences. As this space is separable, by the Skorokhod Representation Theorem, we may work in a probability space $(\Omega, \mathcal{F}, \tilde{\P})$ such that this convergence is almost sure. Note also that $\bP^n_S(\cdot)$ is completely determined by the values of $(\cC_i)_{i \geq 1}$. 

    On this probability space, we almost surely have the convergence 
    \begin{equation*}
        \limn \sum_{i \geq 1} \Big( \frac{|\cC_i|}{n^{2/3}} - |\gamma_i| \Big)^2  = 0,
    \end{equation*}
    and by Cauchy-Schwarz 
    \begin{align*}
        \Big| \sum_{i \geq 1} \Big( \frac{|\cC_i|}{n^{2/3}} \Big)^2  - \sum_{i \geq 1} |\gamma_i|^2 \Big| &= \Big| \sum_{i \geq 1} \Big( \frac{|\cC_i|}{n^{2/3}} - |\gamma_i| \Big) \Big(\frac{|\cC_i|}{n^{2/3}}  + |\gamma_i| \Big) \Big| \\
        &\leq \Big( \sum_{i \geq 1} \Big( \frac{|\cC_i|}{n^{2/3}} - |\gamma_i| \Big)^2 \Big)^{\frac{1}{2}} \cdot \Big( \sum_{i \geq 1} \Big( \frac{|\cC_i|}{n^{2/3}} + |\gamma_i| \Big)^2 \Big)^{\frac{1}{2}} \\
        &\leq \Big( \sum_{i \geq 1} \Big( \frac{|\cC_i|}{n^{2/3}} - |\gamma_i| \Big)^2 \Big)^{\frac{1}{2}} \cdot \Big( 2\sum_{i \geq 1} \Big( \frac{|\cC_i|}{n^{2/3}} \Big)^2 + 2\sum_{i \geq 1} |\gamma_i|^2 \Big)^{\frac{1}{2}},
    \end{align*}
    which almost surely goes to zero as $n \rightarrow \infty$. Therefore, we have that  
    \begin{equation*}
        \limn \frac{\sqrt{\sum_{i \geq 1} |\gamma_i|^2}}{\sqrt{\sum_{i \geq 1} \big( |\cC_i|/n^{2/3} \big)^2 }} = 1
    \end{equation*}
    almost surely. 
    Furthermore, if we set
    \begin{equation*}
        p_{n,i} := \frac{|\cC_i|}{n},
    \end{equation*}
    then the limit
    \begin{equation}
        \limn \frac{p_{n,i}}{s_n} = \limn \frac{|\cC_i|}{\sqrt{\sum_{i \geq 1}|\cC_i|^2 }} = \frac{|\gamma_i|}{\sqrt{\sum_{i \geq 1}|\gamma_i|^2 }} =: \theta_i
    \end{equation}
    exists for any $i \geq 1$.

    Theorem~4 in \cite{CP00} then gives (for a fixed realization of $\omega \in \Omega)$ that for all $r \geq 0$
    \begin{equation*}
       \limn \bP^n_S \big( s_n t_1 >  r \big) = e^{-\frac{1}{2} (1 - \sum_{i\geq 1}\theta_i^2 ) r^2} \prod_{i \geq 1}(1 + \theta_i r) e^{- \theta_i r}.
    \end{equation*}
    As this equality holds for almost all $\omega$ in the probability space $(\Omega, \mathcal{F}, \tilde{\P})$, the convergence in distribution follows.
\end{proof}

We use Theorem~\ref{T:repeat_times} to establish the following lemma.

\begin{lemma} \label{L:repeat_time}
For any $\varepsilon > 0$, there exists $r(\varepsilon), n_0(\varepsilon) > 0$ such that for $n \geq n_0$
\begin{equation} \label{eq:tn1_E_bound}
    \E\Big[ \bP^n_S \Big(t_1 > n^{1/3} r(\varepsilon) \Big) \Big] \geq 1 - \varepsilon.
\end{equation}

\end{lemma}
\begin{proof}
    Recall that we defined in the statement of the last theorem
    \begin{equation*}
        s_n = \frac{1}{n} \sqrt{\sum_{i \geq 1}|\cC_i|^2 } = \frac{1}{n^{1/3}} \sqrt{\sum_{i \geq 1} (|\cC_i|/n^{2/3})^2 } .
    \end{equation*}
    It follows from the convergence of the scaled component sizes to the Brownian excursions in proof of Theorem~\ref{T:repeat_times}, and by the fact that the excursions are almost surely square summable, that there exists a constant $B = B(\varepsilon) > 0$ such that
    \begin{equation*}
        \P( s_1 \leq B n^{-1/3} ) \geq 1 - \varepsilon/3.
    \end{equation*}
    for $n$ large enough. Consequently, for any $\Tilde{r} > 0$
    \begin{equation*}
        \E \Big[ \bP^n_S \big( t_1 \leq \frac{\Tilde{r}}{B} n^{1/3} \big) \Big] \leq  \E \Big[ \bP^n_S \big( s_n t_1 \leq \Tilde{r} \big) \Big] + \varepsilon/3.
    \end{equation*}
    Using Theorem~\ref{T:repeat_times} together with the Dominated Convergence Theorem (note that $\bP^n_S( s_n t_1 \leq \Tilde{r})$ is bounded) we have, for all $n$ large enough (depending on $\varepsilon$)
    \begin{equation} \label{eq:bound_repeat_expectation}
        \E \Big[ \bP^n_S \big( t_1 \leq \frac{\Tilde{r}}{B} n^{1/3} \big) \Big] \leq \tilde{\E}\Big[1 -  e^{-\frac{1}{2} (1 - \sum_{i\geq 1}\theta_i^2 ) \Tilde{r}^2} \prod_{i \geq 1}(1 + \theta_i \Tilde{r}) e^{- \theta_i \Tilde{r}} \Big] + 2\varepsilon/3.
    \end{equation}
    Notice that for any $\Tilde{r} \geq 0$ the function
    \begin{equation*}
        e^{\frac{1}{2} x^2  \Tilde{r}^2}(1 + x \Tilde{r}) e^{- x \Tilde{r}}
    \end{equation*}
    is increasing in $x$, and so as $\theta_i \geq 0$ we obtain 
    \begin{equation*}
        e^{-\frac{1}{2} (1 - \sum_{i\geq 1}\theta_i^2 ) \Tilde{r}^2} \prod_{i \geq 1}(1 + \theta_i \Tilde{r}) e^{- \theta_i \Tilde{r}} \geq e^{-\frac{1}{2} \Tilde{r}^2}.
    \end{equation*}
    In particular, the right-hand side of \eqref{eq:bound_repeat_expectation} is bounded from above by 
    \begin{equation*}
        1 - e^{-\frac{1}{2} \Tilde{r}^2} + 2 \varepsilon/3 .
    \end{equation*}
    The result follows by choosing $\Tilde{r}$ small enough depending on $\varepsilon$, and letting $r(\varepsilon) = \Tilde{r}/B$.
\end{proof}

\begin{remark}
    One may also obtain that the event $\{ t_1 \leq r(\varepsilon) n^{1/3} \}$ occurs with probability at least $1 - \varepsilon$ provided that $r(\varepsilon)$ is large enough. Indeed, with high probability, the largest component $\cC_1$ has size at least $n^{2/3}/A$, and $t_1$ is less than the first time $\cC_1$ is selected twice. This can be bounded by the sum of two independent geometric random variables with means $A n^{1/3}$. 
\end{remark}


\section{Proof of the lower bound in Theorem~\ref{T:diam_extreme}} \label{S:lower_bound}

We start this section by briefly recalling the Aldous-Broder algorithm (see \cite{Bro89,Ald90}) to generate weighted USTs. Let $(X_i)_{i \geq 0}$ be a random walk on the weighted graph $(G, \weight)$. Formally, this is a Markov chain on the vertices of $G$, started at some (arbitrary) vertex $ X_0 = v_0$, with one-step transition probability to jump from a vertex $u$ to a neighboring vertex $v\sim u$ given by
\begin{equation*}
    p(u,v) \coloneqq \frac{\weight(u,v)}{\sum_{x \sim u} \weight(u,x)}.
\end{equation*}
For the Aldous-Broder algorithm, we run the random walk until the first time every vertex is visited. Whenever a previously unvisited vertex is seen by the random walk $X_i$, the Aldous-Broder algorithm adds the previous edge $(X_{i-1}, X_i)$ to the tree. The resulting tree is distributed as the weighted UST as in \eqref{eq:PomegaT}.


\subsection{UST estimates}
We briefly collect two estimates about (weighted) USTs that will be used in the proof of Theorem~\ref{T:diam_extreme}.  In a sequence of works, Aldous \cite{Ald91a, Ald91b, Ald93} introduced a fractal object called the \textit{continuum random tree}, which arises as a universal limiting object of USTs on high-dimensional graphs. In particular, the (limiting) distribution of distances between finitely many points in a UST is well understood. 

We first give the precise formula for the distribution of the distance between two vertices in a UST on the complete graph, which we use to give some (weak) bound on the expected distance between two randomly selected vertices. 

\begin{lemma} \label{L:dist_u_v_UST_Km}
    Let $K_m$ be the complete graph on $m \geq 2$ vertices and denote by $\cT_m$ a UST on $K_m$. If $u,v \in K_m$ are two distinct vertices, then for $1 \leq L \leq m-1$
    \begin{equation} \label{eq:laplacian_walk_distance}
        \bP( d_{\cT_m}(u,v) \geq L) =  \prod_{k=1}^{L-1} \frac{m-k-1}{m} = \frac{1}{m^{L-1}} \frac{(m-2)!}{(m-1-L)!}.
    \end{equation}
    As a consequence, there exists a constant $C > 0$ such that, if $u$ and $v$ are two vertices chosen independently and uniformly at random from $V(K_m)$, then
    \begin{equation} \label{eq:expected_dist_UST_Km}
        \bE[ d_{\cT_m}(u,v) ] \geq C \sqrt{m}.
    \end{equation}
\end{lemma}
\begin{proof}
    Equation \eqref{eq:laplacian_walk_distance} will follow by the Laplacian random walk representation of the loop erased random walk. Denote by $\bP_x$ the law of a random walk $(X_i)_{i\geq 1}$ with $X_1 = x$ and let $\tau_S$ (resp.\ $\tau_S^+$) be the first hitting (resp. return) time of a set $S$ of vertices, where we write $\tau_v\coloneqq \tau_{\{v\}}$. We write $\bP^Y$ for the law of a Laplacian random walk $(Y_i)_{i \geq 0}$ started at $Y_0 = y_0 = u$ and stopped when hitting $v$. Conditional on the previous trajectory, the one-step transition probabilities are
    \begin{equation*}
        \bP^Y \big( Y_k = y_k \mid (Y_0, \ldots, Y_{k-1} ) = (y_0, \ldots, y_{k-1}) \big) = \bP_{y_{k-1}}( X_1 = y_k \mid \tau_{v} < \tau^+_{\{ y_0, \ldots, y_{k-1}\}})
    \end{equation*}
    for $k \geq 1$. The distribution of the (whole) path between $u$ and $v$ is equal to that of the distribution of the Laplacian random walk from $u$ to $v$, see e.g.\ \cite[Exercise 4.1]{LP16}.

    Let $A \subseteq V(K_m)$ and let $a \in A$. On the complete graph, given $x\neq v$ and $x,v \not\in A$ we have that
    \begin{align*}
        \bP_x( \tau_v < \tau_A) &= \frac{1}{|A| + 1}, \\
        \bP_a( \tau_v < \tau^+_A) &= \frac{1}{m-1} + \frac{1}{m-1} \sum_{x \not\in \{v\} \cup A} \bP_x( \tau_v < \tau_A) \\
        &= \frac{1}{m-1} + \frac{m - |A| -1 }{(m-1)(|A| + 1)} = \frac{m}{m-1} \cdot \frac{1}{|A| + 1},
    \end{align*}
    and hence for $y_k \not\in \{y_1, \ldots, y_{k-1}, v\}$
    \begin{align*}
        \bP_{y_{k-1}}( X_1 = y_k  \mid \tau_{v} < \tau^+_{\{ y_0, \ldots, y_{k-1}\}}) &= \frac{ \frac{1}{m -1 } \cdot \bP_{y_k} ( \tau_{v} < \tau_{\{ y_0, \ldots, y_{k-1}\}} )}{ \bP_{y_{k-1}}(\tau_v < \tau^+_{\{ y_0, \ldots, y_{k-1}\}}) } = \frac{1}{m}.
    \end{align*}
    For an arbitrary trajectory $(Y_0, \ldots Y_{k-1})$ of the Laplacian random walk (with $Y_0 = u$, $Y_i \neq v$ for each $i \leq i-1$), it then follows that
    \begin{equation*}
        \bP^Y \big(Y_{k} \neq v \mid (Y_0, \ldots, Y_{k-1}) \big) 
         = \frac{m -(k+1)}{m}
    \end{equation*}
    since there are $m-(k+1)$ many vertices in $V(K_m) \setminus \{y_1, \ldots, y_{k-1}, v \}$. This readily gives \eqref{eq:laplacian_walk_distance} since the distance $d_{\cT_m}(u,v)$ is at least $L$ if and only if the Laplacian random walk has not terminated in the $(L-1)$-step.

    To deduce the inequality \eqref{eq:expected_dist_UST_Km}, notice that the event $\{u \neq v\}$ occurs with probability $1- 1/m$, so that
     \begin{equation*}
        \bP( d_{\cT_m}(u,v) \geq L) \geq \bP( d_{\cT_m}(u,v) \geq L \mid u \neq v)(1- 1/m).
    \end{equation*}
    Using Stirling's approximation one can show that the logarithm applied to the right hand side of \eqref{eq:laplacian_walk_distance} equals
    \begin{multline*}
        - (L-1) \log (m) + (m-2) \log (m-2)  - (m-2) + \frac{1}{2} \log (2\pi (m-2) ) \\
        - (m-1-L) \log(m-1-L) + (m-1 - L) - \frac{1}{2} \log (2\pi (m-1 - L) ) - O\big( \frac{1}{m} \big)
    \end{multline*}
    When $L$ is of order $\sqrt{m}$ the above expression is of constant order\footnote{For larger $L$, the right hand side of \eqref{eq:laplacian_walk_distance} decays like $\exp(-L^2/m)$. Hence, the distances between fixed pairs of vertices do not get much larger than $\sqrt{m}$ either.}. Thus, for distinct vertices $u$ and $v$
    \begin{equation*} 
        \bP( d_{\cT_m}(u,v) \geq L) \geq C
    \end{equation*} 
    from which \eqref{eq:expected_dist_UST_Km} follows.
\end{proof}
Kirchhoff's formula for USTs (see e.g.\ \cite[Section 4.2]{LP16}) states that 
\begin{equation*}
    \bP_\cT \big( \{u,v\} \in \cT \big) = \weight(u,v) \effR{}{u}{v},
\end{equation*}
where $\effR{}{u}{v}$ is the effective resistance between $u$ and $v$. Whenever $\gamma$ is chosen large enough, the main contribution to the effective resistance between two vertices in the same component will come from the edges in the component. Using a coupling between an \ErdosRenyi random graph and the realization of the weights as briefly described in Section~\ref{SS:ER-graphs}, i.e., the retained edges correspond to edges $e$ with weight $\weight(e) = n^{1 + \gamma}$ so that
\begin{equation*}
    \cC(u) = \big\{ v \in V : \exists \text{ a path between $u$ and $v$ of edges $e$ with } \weight(e) = n^{1+\gamma} \big\}. 
 \end{equation*}
We will make use of the following lemma, which is essentially the same as Lemma~4.3 of \cite{MSS24}.
\begin{lemma}[cf.\ {\cite[Lemma~4.3]{MSS24}}] \label{L:UST_path_inside}
    Denote by $\pi_\cT(u,v)$ the path in the tree $\cT$ connecting the vertices $u$ and $v$. Then
    \begin{equation*}
        \bP_\cT \Big( \exists u,v \in V(K_n), e \in E(K_n) \text{ with } v \in \cC(u), e \in \pi_\cT(u,v), \weight(e) = 1 \Big) \leq n^{4-\gamma}.
    \end{equation*}
\end{lemma}
\begin{proof}[Proof sketch]
    Suppose that $e = \{u,v\}$ is an edge with weight $\weight(e) = 1$, and that $u$ and $v$ are in the same component. Then there is a path of length at most $n$ consisting of edges with weights equal to $n^{1+\gamma}$, so that by Kirchhoff's formula
    \begin{equation*}
         \bP_\cT \big( e \in \cT \big) = \weight(e) \effR{}{u}{v} \leq n \frac{1}{n^{1+\gamma}} = n^{-\gamma}.
    \end{equation*}
    Extending this to all pairs of vertices and all edges may be done in the same way as in the proof of Lemma~4.3 in \cite{MSS24}.
\end{proof}

\subsection{Contraction of components} \label{SS:contract}
Lemma \ref{L:UST_path_inside} establishes that, if $\gamma \geq 5$, then with probability at least $1 - n^{-1}$ any two vertices inside the same component are connected by a path that stays inside said component. In this case, the tree $\cT$ consists of a union of subtrees $T_1, \ldots, T_k$ on the components $\cC_1, \ldots, \cC_k$, respectively, joined together by edges with weight equal to $1$. Given any realization of the subtrees $T_1, \ldots, T_k$ on the components, by the Spatial Markov property of the UST (see e.g.\ \cite[Section 2.2.1] {HN19} for more details), we may contract each subtree $T_i$ into a single vertex to obtain a new graph $G' = G'(T_1, \ldots, T_k)$ (with weights $\weight'$ inherited from $\weight$) such that for any $F \subseteq E$
\begin{equation} \label{eq:Spatial_Markov_Tk}
        \bP^\weight_G \big( \cT = F \mid \bigcup_{i=1}^k T_i \subseteq \cT \big) = \bP^{\weight'}_{G'} \big( \cT \cup \bigcup_{i=1}^k T_i = F \big),
\end{equation}
where we identified $\cT$ with its set of edges. Furthermore, each vertex $A \in V(G')$ corresponds to some connected component $\cC_A$ in $G$, and the weight between any two distinct vertices $A,B \in V(G')$, corresponding to two disjoint components $\cC_A, \cC_B$ in $G$, is given by 
\begin{equation}
    \weight'(A,B) = \sum_{a \in \cC_A, b \in \cC_B} \weight(a,b) = \sum_{a \in \cC_A, b \in \cC_B} 1 = |\cC_A| |\cC_B|. \label{eq:weight_contracted}
\end{equation}
To make a coupling between $(X_i)_{i\geq 1}$ and the process of picking components $(S_i)_{i \geq i}$ from Section~\ref{S:size_biased_GNP} simpler, we replace all self-loops centered at a vertex $A \in V(G')$ (obtained from the contraction) with a single self-loop of weight $|\cC_A|^2$. This leaves the law of the UST unchanged (as self-loops never appear in the UST), and it makes \eqref{eq:weight_contracted} also hold for the case when $A = B$.

We may run a random walk $(X_i)_{i \geq 1}$ on $G'$ by starting at a random vertex $X_1$ corresponding to uniformly choosing a vertex in $G$ and taking the contracted component. If we denote by $\bP_X$ the corresponding law of the walk $(X_i)_{i\geq 1}$, then for any vertices $A$ and $B$ in the contracted graph $G'$ corresponding to components $\cC_A$ and $\cC_B$, we have by \eqref{eq:weight_contracted} that
    \begin{align*}
        \bP_X( X_{k+1} = B \mid X_{k} = A) &= \frac{|\cC_A| |\cC_B|}{\sum_{W \in V(G')} |\cC_A| |\cC_W| }  \\
        &= \frac{|\cC_A| |\cC_B|}{ |\cC_A| n } = \frac{|\cC_B|}{n}.
     \end{align*}
The following lemma now follows immediately.
\begin{lemma} \label{L:couple_X_S}
    There exists a coupling between the random walk $(X_i)_{i \geq 1}$ on $G'$ and the components $(S_i)_{i \geq 1}$ from Section~\ref{S:size_biased_GNP} such that $X_i = S_i$ almost surely for all $i \geq 1$.
\end{lemma}

As it will be clear later, all of our proofs hold uniformly over all possible realizations of the subtrees $T_i$ in the components $\cC_i$ (in fact, most components are already trees themselves). We will therefore drop the notation regarding the $T_i$'s and the weights $\weight'$, and write $\bP_{\cT'}$ for the law of the tree on $G'$ conditional on a realization of the subtrees $T_1, \ldots, T_k$.

\subsection{Proof of lower bound in Theorem~\ref{T:diam_extreme}} \label{SS:lower_idea}

The proof idea behind the lower bound of Theorem~\ref{T:diam_extreme} is the following. Using \eqref{eq:Spatial_Markov_Tk}, we first contract the components into single vertices, and then we run a random walk $X'$ on the resulting graph $G'$. The diameter of the tree $\cT'$ on $G'$ is at least as large as the first time the random walk $X'$ revisits a vertex in $G'$, which corresponds to revisiting a component in $G$. The results of Section~\ref{S:size_biased_GNP} will imply that the random walk visits approximately $n^{1/3}$ many different components before creating a loop. Each visited component $\cC_i$ is almost\footnote{If $\cC_i$ is cycle-free, then it is exactly distributed as a UST.} distributed as a UST on $|\cC_i|$ many vertices, and hence it will contribute approximately $\sqrt{|\cC_i|}$ to the diameter to the tree. The result will follow by showing that such a sum consisting of $n^{1/3}$ many terms concentrates well enough around its mean, provided that we truncate each term at a certain level which polynomial in $n$ (we choose $n^{1/20}$ but any exponent small enough would suffice). We start by comparing the tree to the components selected in Section~\ref{S:size_biased_GNP}.

\begin{lemma} \label{L:compare_T_S}
    Let $\cT$ be a random spanning tree with law as in \eqref{eq:PomegaT}, and denote by $S_1, S_2, \ldots$ components chosen at random as in Section~\ref{S:size_biased_GNP}. If $t_1$ denotes the first repeat time of the process $(S_i)_{i\geq 1}$, then with $\bP_\cT$-probability at least $1 - n^{-1}$
     \begin{equation} \label{eq:construction_lower}
        \diam(\cT) \geq \sum_{i=1}^{t_1 - 1} d_{S_i}(u_i, v_i),
    \end{equation}
    where $u_i, v_i$ are two uniform random points in $S_i$ and $d_{S_i}$ is the graph distance on $S_i$. 
\end{lemma}

\begin{proof}
    We consider the law $\bP_{\cT'}$ from \eqref{eq:Spatial_Markov_Tk} on $G'$, which consists of contracting the connected components into single vertices, instead of $\bP_{\cT}$. Any vertex $A \in V(G')$ corresponds to a unique component $\cC_A$ of the critical \ErdosRenyi random graph.  We will show that, after uncontracting $\cT'$ to obtain $\cT$, \eqref{eq:construction_lower} holds almost surely, so that by the argument in Section~\ref{SS:contract}, the event fails with probability at most $n^{-1}$.

    To obtain $\cT'$ on $G'$ using the Aldous-Broder algorithm, we run a random walk $(X_i)_{i \geq 1}$ on $G'$ started by choosing a vertex $u_1$ uniformly at random and taking $X_1$ to be the component $\cC(u_1)$ corresponding to $u_1$. By Lemma~\ref{L:couple_X_S} we may couple this process with the components $(S_i)_{i \geq 1}$ from Section~\ref{S:size_biased_GNP} such that $X_i = S_i$. Note that between distinct vertices $A,B \in V(G')$, there are $|\cC_A| |\cC_B|$ parallel edges (all with weight equal to one) which the walk may use. As each edge is equally likely to be traversed by the random walk from $X_i = A$ to $X_{i+1} = B$, and every pair $v \in \cC_A$, $u \in \cC_B$ corresponds to exactly one edge, the random walk exits $\cC_A$ from a uniformly randomly chosen point $v_i \in \cC_A$ and enters $\cC_B$ at a uniformly randomly chosen point $u_{i+1} \in \cC_B$.
    
    Now recall that $t_1 \geq 2$ was the first time a component in the sequence $(S_i)_{i \geq 1}$ was repeated. Note that under this coupling, until time $t_1 - 1$, no cycles have appeared for the random walk. By the Aldous-Broder algorithm, this means that the edges $(X_i, X_{i+1})$, $1 \leq i \leq t_1 - 2$ are in the spanning tree $\cT'$. Each edge $(X_i, X_{i+1})$ connects two components $S_i$ and $S_{i+1}$ (with corresponding trees $T(S_i)$ obtained from the contraction procedure as in \eqref{eq:Spatial_Markov_Tk}) via random vertices $v_i$ and $u_{i+1}$. Furthermore, every pair of vertices $u_i$ and $v_i$ is connected in $\cT$ by some (possibly empty) path in $T(S_i)$. Since the distances in the contracted trees $T(S_i)$ cannot be smaller than the distances in the components they correspond to, we obtain in the uncontracted graph that
    \begin{equation*}
        \diam(\cT) \geq t_1 - 2 + \sum_{i=1}^{t_1 - 1} d_{T(S_i)}(u_i, v_i) \geq \sum_{i=1}^{t_1 - 1} d_{S_i}(u_i, v_i),
    \end{equation*}
    completing the proof. 
\end{proof}   

\begin{remark}   
   Notice that the bound in \eqref{eq:construction_lower} is independent of the specific subtrees $T_1, \ldots, T_k$ chosen in the contraction argument to obtain \eqref{eq:Spatial_Markov_Tk}.
\end{remark}

In order to obtain a better concentration for the sum in \eqref{eq:construction_lower}, we further lower bound the diameter by taking the minimum between $n^{1/20}$ and each term in the sum; that is, 
    \begin{equation} \label{eq:construction_lower_wedge}
        \diam(\cT) \geq \sum_{i=1}^{t_1 - 1} \Big( d_{S_i}(u_i, v_i) \wedge n^{1/20} \Big).
    \end{equation}
(We remark that any exponent strictly smaller than $1/15$ would suffice.) We will show that the right-hand side of \eqref{eq:construction_lower_wedge} concentrates well enough.

    \begin{proof}[Proof of the lower bound in Theorem~\ref{T:diam_extreme}]

    \smallskip
    
    We define the following (good) events for some $B,r >0$ to be determined later:
    \begin{align}
        \mathcal{G}_1 &\coloneqq\{ t_1 \geq r n^{1/3} + 1 \}, \nonumber  \\
        \mathcal{G}_2 &\coloneqq\Big\{ \bE^n_S \Big[\sum_{i=1}^{r n^{1/3}} \big( d_{S_i}(u_i, v_i) \wedge n^{1/20} \big) \Big]   \geq \frac{1}{B} r n^{1/3} \log n \Big\} \nonumber \\
        &= \Big\{ \bE^n_S \big[  d_{S_1}(u_i, v_i) \wedge n^{1/20} \big] \geq \frac{1}{B} \log n \Big\}.  \label{eq:lower_good_2}
    \end{align}
    In Lemma~\ref{L:repeat_time}, we showed that, if $r$ is chosen small enough, then for all large enough $n$ we have
    \begin{equation*}
        \E \big[ \bP^n_S(\mathcal{G}_1^c) \big] \leq \frac{\varepsilon}{3}.
    \end{equation*}
    In order to simplify the forthcoming notation, in what follow we will write $\bP_S$ instead of $\bP^n_S$. In particular, Lemma~\ref{L:compare_T_S} and the inequality in \eqref{eq:construction_lower_wedge} together yield that
     \begin{multline}
        \E\Big[ \bP_\cT \big(\diam(\cT) \geq \frac{1}{A} n^{1/3} \log n \big) \Big] \\ 
        \geq \E\Big[ \bP_S \big( \sum_{i=1}^{r n^{1/3}} \big( d_{S_i}(u_i, v_i) \wedge n^{1/20} \big) \geq \frac{1}{A} n^{1/3} \log n \big) \Big] - \frac{\varepsilon}{3} - \frac{1}{n}. \label{eq:good_event_1}
    \end{multline}
    
    We claim that, on $\mathcal{G}_2$, the event in \eqref{eq:good_event_1} occurs with high enough probability. With this in mind, applying Paley–Zygmund's inequality (w.r.t.\ to $\bP_S$) to the non-negative random variable
    \begin{equation*} 
        M := \sum_{i=1}^{r n^{1/3}} \big( d_{S_i}(u_i, v_i) \wedge n^{1/20} \big),
    \end{equation*}
    we obtain for $\theta \in [0,1]$
    \begin{equation} \label{eq:paley-zyg}
        \bP_S \big( M > \theta \bE_S[M] \big) \geq (1- \theta)^2 \frac{\bE_S[M]^2}{\bE_S[M^2]}.
    \end{equation}
    The second moment of $M$ can be bounded by
    \begin{align}
        \bE_S[M^2] &= \sum_{i,j=1}^{r n^{1/3}} \bE_S \big[ (d_{S_i}(u_i, v_i) \wedge n^{1/20}) (d_{S_j}(u_j, v_j) \wedge n^{1/20}) \big] \nonumber \\
        &\leq \sum_{i,j=1, i \neq j}^{r n^{1/3}} \bE_S \big[ d_{S_i}(u_i, v_i) \wedge n^{1/20} \big] \bE_S \big[ d_{S_j}(u_j, v_j) \wedge n^{1/20} \big] + \sum_{i=1}^{r n^{1/3}} \bE_S[ n^{1/10}] \nonumber \\
        &\leq \bE_S[ M]^2 + r n^{1/3 + 1/10}. \label{eq:dia_lower_2nd_mom}
    \end{align}

    Hence, we have
    \begin{align*}
        \mathbbm{1}_{\mathcal{G}_2} \, \bE_S[M]  &\geq \mathbbm{1}_{\mathcal{G}_2}\, \frac{1}{B} n^{1/3} \log n,\\
        \mathbbm{1}_{\mathcal{G}_2}  \cdot 1/\bE_S[ M^2] &= \mathbbm{1}_{\mathcal{G}_2} \cdot (1+o(1)) /\bE_S[ M]^2,
    \end{align*}
    where the second equality follows from \eqref{eq:dia_lower_2nd_mom}. Therefore, 
    \begin{equation*}
        \mathbbm{1}_{\mathcal{G}_2}\, \bP_S \big( M \geq \theta \frac{1}{B} n^{1/3} \log n \big) \geq \mathbbm{1}_{\mathcal{G}_2}\, (1+o(1)) (1-\theta)^2.
    \end{equation*}
    By choosing $\theta = \varepsilon/5$ and $A = B/\theta$, the inequality of \eqref{eq:good_event_1} gives for $n$ large enough that
    \begin{equation*}
        \E\Big[ \bP_\cT \big(\diam(\cT) \geq \frac{1}{A} n^{1/3} \log n \big) \Big] \geq 1 - \frac{2 \varepsilon}{3}  - \P(\mathcal{G}_2^c).
    \end{equation*}
    To complete the proof, it suffices to show that $\mathcal{G}_2^c$ occurs with probability at most $\varepsilon/3$, provided that $B = B(\varepsilon, r)$ is large enough. 
    
    To this end, we will again provide first and second moment estimates (now with respect to $\P$ instead of $\bP_S$ as before) on $\bE_S[ d_{S_i}(u_i, v_i) \wedge n^{1/20}]$. Recall that $\cC(x)$ is the component of a vertex $x$, and that the sampling procedure of the components $S_i$ is equivalent to first uniformly at random selecting a vertex $x$ and taking its component $\cC(x)$. Let $u_x, v_x$ be two vertices chosen uniformly from $\cC(x)$. By exchangeability of the vertices (see also \eqref{eq:E_f(S)})
    \begin{align*}
        \E\Big[ \bE_S\big[d_{S_i}(u_i,v_i) \wedge n^{1/20} \big] \Big] &= \E\Big[ \sum_{x \in V} \frac{1}{n} d_{\cC(x)}(u_x,v_x) \wedge n^{1/20} \Big] \\
        &= \E \big[ d_{\cC(1)}(u_1, v_1) \wedge n^{1/20} \big].
    \end{align*}
    Conditional on $\cC(1)$ containing no cycles, the component $\cC(1)$ is distributed as a uniform spanning tree on a complete graph with $|\cC(1)|$ many vertices. Indeed, the probability of any configuration of $\cC(1)$ without cycles occurring is proportional to $p^{|\cC(1)|}$.     
    Therefore, applying Lemmas~\ref{L:dist_u_v_UST_Km} and~\ref{L:collection_RG_facts} we obtain
    \begin{align}
        &\E[d_{\cC(1)}(u_1,v_1) \wedge n^{1/20}]  \nonumber \\ 
        &\geq \sum_{j=1}^{n^{1/20}}  \E \big[ d_{\cC(1)}(u_1,v_1) \bigm|  \cC(1) \text{ cycle--free}, |\cC(1)| = j \big] \hspace{2pt} \P \big(\cC(1) \text{ cycle--free}, \cC(1) = j \big) \nonumber \\
        &\geq c \sum_{j=2}^{n^{1/20}} \sqrt{j} \hspace{2pt} \P\big( \cC(1) \text{ cycle--free}, |\cC(1)| = j \big) \nonumber \\
        &= c \sum_{j=2}^{n^{1/20}} \sqrt{j} \hspace{2pt} \P \big(\cC(1) \text{ cycle--free} \bigm| |\cC(1)| = j \big) \hspace{2pt} \P \big(|\cC(1)| = j \big) \nonumber \\
        \nonumber&\geq  c' \sum_{j=2}^{n^{1/20}} \sqrt{j} \hspace{2pt} \frac{1}{j^{3/2}}\\
        \nonumber&\geq c''n^{1/40}\\
        &\geq c'' \log n, \label{eq:distance_1st_moment_lower}
    \end{align}
    where $c, c', c'' > 0$ are some constants independent of $n$.
    
    On the other hand, in order to bound the second moment from above, we first split the sum into two parts depending on whether the components of $x$ and $y$ are distinct or not, to obtain 
    \begin{align}
        & \vspace{-1cm }\E\Big[ \bE_S \big[  d_{S_i}(u_i,v_i) \wedge n^{1/20} \big]^2 \Big] \nonumber \\
        &= \E \Big[\frac{1}{n^2} \sum_{x,y \in V} (d_{\cC(x)}(u_x,v_x) \wedge n^{1/20}) (d_{\cC(y)}(u_y,v_y) \wedge n^{1/20}) \Big] \nonumber \\
        &\leq  \frac{1}{n^2} \Big( \sum_{x,y \in V} n^{1/10} \P \big(\cC(x) = \cC(y) \big)  \nonumber \\
        &\quad \qquad + \E\big[ (d_{\cC(x)}(u_x,v_x) \wedge n^{1/20}) (d_{\cC(y)}(u_y,v_y) \wedge n^{1/20} ) \mathbbm{1}_{\cC(x) \neq \cC(y)} \big] \Big). \label{eq:diam_2nd_moment}
    \end{align}
    By Lemma~\ref{L:same_component} we know that $\P( \cC(x) = \cC(y)) \leq c n^{-1/3}$ whenever $x \neq y$, so that the first term in \eqref{eq:diam_2nd_moment} can be further bounded from above by
    \begin{equation*}
    \frac{n^{1+1/10} + c \cdot n^{2 +1/10 - 1/3}}{n^{2}} = O(1).
    \end{equation*}
    Concerning the second term in \eqref{eq:diam_2nd_moment}, we control it by conditioning on the realization of the component $\cC(y)$ and its size. Namely, for $x,y \in V$ let $\mathcal{A}^x_y$ be the collection of sets $A \subset [n] \setminus \{x\}$ with $y \in A$. Then, conditioning on the realization of the component $\cC(y)$, we have
    \begin{align}
        &\E\big[ (d_{\cC(x)}(u_x,v_x) \wedge n^{1/20}) (d_{\cC(y)}(u_y,v_y) \wedge n^{1/20}) \mathbbm{1}_{\cC(x) \neq \cC(y)} \big] \nonumber \\
        & \leq \sum_{A \in \mathcal{A}^x_y, |A| \leq n^{1/4}} \Big( \E\big[ (d_{\cC(x)}(u_x,v_x) \wedge n^{1/20}) (d_{\cC(y)}(u_y,v_y) \wedge n^{1/20}) \mid \cC(y) = A \big] \nonumber \\
        & \hspace{5cm} \cdot \P\big( \cC(y) = A  \big) \Big) + n^{1/10} \P\big( |\cC(y)| \geq n^{1/4} \big). \label{eq:product_distance}
    \end{align}
    Conditional on the event $\{ \cC(y) = A\}$, the distances in $\cC(y)$ depend only on edges with both endpoints in $A$, and the distances in $\cC(x)$ depend only on edges with both endpoints in $A^c$, so that the two terms in the (conditional) expectation of \eqref{eq:product_distance} are conditionally independent. Furthermore, given that $ \cC(y) = A$, the distances of $\cC(x)$ in $\mathbb{G}(n,1/n)$ are distributed as the distances of $\cC(x)$ in $\mathbb{G}(n - |A|, 1/n)$ with corresponding law $\P_{n-|A|,1/n}$ (where by abuse of notation we may need to relabel the vertices such that all have labels inside $[n - |A|]$). Using these two facts and Lemma~\ref{L:RG_typical} for $\P(\cC(y) \geq n^{1/4})$, the right hand side of \eqref{eq:product_distance} is upper bounded by
    \begin{align}
        &\sum_{A \in \mathcal{A}^x_y, |A| \leq n^{1/4}} \Bigg( \E \big[ d_{\cC(y)}(u_y,v_y) \wedge n^{1/20} \mid \cC(y) = A \big] \P\big( \cC(y) = A  \big) \nonumber \\
        & \qquad \qquad \qquad \qquad \qquad \cdot \E_{n-|A|,\frac{1}{n}} \big[ d_{\cC(x)}(u_x,v_x) \wedge n^{1/20} \big] \Bigg) +  n^{1/10} \frac{C}{(n^{1/4})^{1/2}} \nonumber \\
        &\leq  \max_{1 \leq b \leq n^{1/4}} \E_{n-b,\frac{1}{n}} \big[ d_{\cC(x)}(u_x,v_x) \wedge n^{1/20} \big]  \E \big[ d_{\cC(y)}(u_y,v_y) \wedge n^{1/20} \big] + o(1). \label{eq:2nd_moment_max}
    \end{align}
    Now applying Lemmas~\ref{L:TV_fact} and~\ref{L:d_TV_remove} we get
    \begin{align*}
        & \max_{1 \leq b \leq n^{1/4}} \E_{n-b,\frac{1}{n}} \big[ d_{\cC(x)}(u_x,v_x) \wedge n^{1/20} \big] \\
        &\qquad \qquad \leq \E_{n,\frac{1}{n}} \big[ d_{\cC(x)}(u_x,v_x) \wedge n^{1/20} \big] + C n^{1/20 + 1/4 - 1/3} \\
        & \qquad \qquad = (1+o(1)) \E_{n,\frac{1}{n}} \big[ d_{\cC(x)}(u_x,v_x) \wedge n^{1/20} \big]
    \end{align*}
    Therefore, by \eqref{eq:diam_2nd_moment}, \eqref{eq:product_distance} and \eqref{eq:2nd_moment_max} we obtain 
    \begin{equation*}
        \E\Big[ \bE_S \big[  d_{S_i}(u_i,v_i) \wedge n^{1/20} \big]^2 \Big] = (1+o(1)) \E\Big[ \bE_S \big[  d_{S_i}(u_i,v_i) \wedge n^{1/20} \big] \Big]^2.
    \end{equation*}
    The proof may now be completed by using the first moment lower bound in \eqref{eq:distance_1st_moment_lower} together with Paley–Zygmund's inequality, to give that $\P(\mathcal{G}_2) \geq 1 - \varepsilon/3$ whenever $B$ (as in \eqref{eq:lower_good_2}) is large enough. 
\end{proof}


\section{Upper bound} \label{S:upper_bound}
We now proceed to the proof of the upper bound in Theorem~\ref{T:diam_extreme}. In the first subsection, we collect some preliminary facts that will be needed in the proof of the upper bound.

\subsection{Preliminary bounds} \label{SS:Upper_pre}

 Denote by $\Exc(H) \coloneqq |E(H)| - |V(H)|$ the excess of a graph $H$. Note that a tree has excess equal to $-1$. We shall make use of the following lemma.
\begin{lemma} \label{L:excess_diam}
    Let $T$ be a spanning tree of a connected graph $H$ with $\Exc(H) = k$ for $k \geq -1$. Then
    \begin{equation}
        \diam(T) \leq 2 (k+2) \diam(H) + k+1.
    \end{equation}
\end{lemma}
\begin{proof}
    Let $T_1 = T$, and let $T_2$ be a spanning tree with $\diam(T_2) \leq 2 \cdot \diam(H)$. Such a tree always exists, as we may take an arbitrary root $v_0$ and let $T_2$ be the breadth-first-search tree. We then have for any vertices $u,v \in V(H)$ that
    \begin{equation*}
        d_{T_2}(u,v) \leq d_{T_2}(u, v_0) + d_{T_2}(v_0, v) = d_{H}(u, v_0) + d_{H}(v_0, v) \leq 2 \cdot \diam(H)
    \end{equation*}
    as required. Note that the edge set $E(T_2) \setminus E(T_1)$ contains at most $k+1$ many edges, as both trees select $|V(H)|-1$ edges from a set of $|E(H)| = |V(H)|+k$ many edges. Lemma~4.10 of \cite{MSS23} then gives that
    \begin{equation*}
        \diam(T_1) \leq (k+2) \diam(T_2) + k+1,
    \end{equation*}
    from which the assertion follows.
\end{proof}

Next, we bound the diameter of a component $\cC(x)$ by comparing it to a branching process tree, which we shall denote by $T_{\BP}$. For the following lemma, recall that $\height(T_{\BP})$ denotes the maximum generation to which the branching process survived,  i.e., the largest $\ell$ such that there is at least one individual at generation $\ell$.

\begin{lemma} \label{L:diam_gen_BP}
    There exists a tree $T_{\BP}$ associated to a branching process with offspring law $\text{Bin}(n,1/n)$, such that 
    \begin{equation*}
        \diam \big( \cC(x) \big) \leq 2 \height(T_{\BP}).
    \end{equation*}
\end{lemma}
\begin{proof}
    Consider starting the exploration process described Section \ref{S:RG_estimates} at vertex $x$, and explore vertices in a breadth-first-search fashion. Each vertex $u$ discovered during the procedure has at most $\text{Bin}(n,1/n)$ neighbors. Thus we can couple $\mathcal{C}(x)$ with a branching process tree $T_{\BP}$, rooted at $x$ with offspring distribution $\text{Bin}(n,1/n)$, in such a way that 
    \begin{equation*}
        |\partial B(x,j)|\leq |\mathcal{L}_j| \text{ for each } j,
    \end{equation*}
    where $\partial B(x,j)$ is the set of vertices at distance $j$ from $x$, whereas $\mathcal{L}_j$ is the set of vertices at distance $j$ from the root (i.e., $x$) in $T_{\BP}$. If $\diam(\mathcal{C}(x))=k$, then there are two vertices $u,v$ in $\mathcal{C}(x)$ such that $d(u,v)=k$. By the coupling, these two vertices cannot be at distance smaller than $k$ in $T_{\BP}$, whence $\diam(T_{\BP}) \geq k = \diam(\mathcal{C}(x))$. The lemma follows by noting that $\diam(T_{\BP}) \leq 2\height(T_{\BP})$.
\end{proof}

Notice that the branching process in Lemma~\ref{L:diam_gen_BP} is \textit{critical} (i.e., it has mean equal to one) and the offspring distribution has variance $1-1/n$. In particular, there exists some universal constant $C$ (independent of $n$) such that for all $k \geq 1$
\begin{equation} \label{eq:height_BP}
    \P \big( \height(T_{\BP}) > k \big) \leq \frac{C_{\BP}}{k},
\end{equation}
see e.g.\ Theorem 12.7 in \cite{LP16}.
Recall from Section~\ref{S:size_biased_GNP} the law $\bP_S$ (with expectation $\bE_S$) that selects a component with probability proportional to its size.
\begin{corollary} \label{Cor:gen_moments}
    There exists a universal constant $C > 0$ such that, for $n \geq 2$ and any $B\geq 1$,
    \begin{align}
        \E \Big[ \bE_S \big[\diam(S_1) \wedge (Bn^{1/3}) \big] \Big] &\leq C \log(Bn), \label{eq:gen_first_moment}\\
        \E \Big[ \bE_S \big[ \big( \diam(S_1) \wedge (Bn^{1/3}) \big)^2 \big] \Big] &\leq C B n^{1/3} \label{eq:gen_second_moment}.
    \end{align}
\end{corollary}

\noindent We will make use of the following two (well-known) identities which express moments of non-negative integer-valued random variables as sums of tail probabilities.
\begin{lemma} \label{L:moments_as_tails}
    If $X \geq 0$ is a random variable supported on $\N_0$, then
    \begin{align*}
        \E[X] &= \sum_{k=0}^\infty \P( X > k), \\
        \E[X^2] &= \sum_{k=0}^\infty (2k + 1) \P(X > k),
    \end{align*}
\end{lemma}
\begin{proof}
    The equalities follow by writing
    \begin{align*}
        X &= \sum_{k=0}^{X-1} 1 = \sum_{k=0}^{\infty} \mathbbm{1}_{\{X > k\}}, \\
        X^2 &= \sum_{k=0}^{X-1} (2k + 1) = \sum_{k=0}^{\infty} (2k+1) \mathbbm{1}_{\{X > k\}},
    \end{align*}
    and taking expectations on both sides.
\end{proof}

\begin{proof}[Proof of Corollary~\ref{Cor:gen_moments}]
    Denote by $T_{\BP}$ the branching process tree from Lemma~\ref{L:diam_gen_BP}. Then by Lemma~\ref{L:diam_gen_BP} and the height bound of \eqref{eq:height_BP}, we have by Lemma~\ref{L:moments_as_tails}
    \begin{align*}
        \E \big[ \bE_S[\diam(S_i) \wedge (Bn^{1/3})] \big] &= \frac{1}{n} \sum_{x \in V} \E\big[ \diam(\cC(x)) \wedge (Bn^{1/3}) \big] \\
        &\leq 2 \E \big[ \height(T_{\BP}) \wedge (Bn^{1/3}) \big] \\
        &= 2 \sum_{k=0}^\infty \P \big( \height(T_{\BP}) \wedge (Bn^{1/3}) > k \big) \\
        &\leq 2 + 2 C_{\BP}  \sum_{k=1}^{Bn^{1/3}} \frac{1}{k},
    \end{align*}
    from which \eqref{eq:gen_first_moment} follows. Furthermore,
     \begin{align*}
        \E \Big[ \bE_S \big[ \big( \diam(S_i) \wedge (Bn^{1/3}) \big)^2 \big] \Big] &=  \frac{1}{n} \sum_{x \in V} \E \big[ \big( \diam(\cC(x)) \wedge (Bn^{1/3}) \big)^2 \big]\\
        &\leq 4 \E \big[ \big( \height(T_{\BP}) \wedge (B n^{1/3}) \big)^2 \big] \\
        &= 4\sum_{k=0}^\infty (2k+1) \P\big( \height(T_{\BP}) \wedge (B n^{1/3}) > k\big) \\
        &= 4 + 4 C_{\BP} \sum_{k=1}^{B n^{1/3}} \frac{2k+1}{k},
    \end{align*}
    giving \eqref{eq:gen_second_moment} whenever $C$ is large enough.
\end{proof}

\subsection{Proof of upper bound in Theorem~\ref{T:diam_extreme}}

Recall from Section~\ref{SS:proof_ideas} the rough outline of the proof strategy for the upper bound in Theorem~\ref{T:diam_extreme}.  By the same reasoning used in the proof of the lower bound of Theorem~\ref{T:diam_extreme}, we will work with the UST measure on $G'$ 
obtained by contracting all the components into single vertices as in \eqref{eq:Spatial_Markov_Tk}. By Lemma~\ref{L:UST_path_inside}, this yields at most an additive error of $n^{-1}$ in any $\bP$-probability bound (we implicitly hide this in an $o(1)$ term later). Further, for a component $\cC$, let $T_{\textrm{d-max}}(\cC)$ be a spanning tree with maximal diameter among all spanning trees of $\cC$. This means that for any realization of the trees $T_k$ in \eqref{eq:Spatial_Markov_Tk}, we have
\begin{equation*}
    \max_{x,y \in \cC} d_\cT(x,y) \leq \diam( T_{\textrm{d-max}} (\cC) )
\end{equation*}
for all components $\cC$. We let $k\coloneqq  \lceil (\log n)^3 \rceil$ and define
\begin{equation} \label{eq:union_largest}
    \mathcal{L} = \bigcup_{i=1}^k \cC_i
\end{equation}
as the union of the $k$ largest components. Here $k$ is chosen large enough such that $|\mathcal{L}| \geq n^{2/3} \log n$, which will imply that the set $\mathcal{L}$ is hit by a random walk within approximately $n^{1/3}$ steps. The logarithmic factor is crucial, since it allows us to control probabilities uniformly via a union bound.

\begin{lemma} \label{L:cT_exc_dia_C1}
    Let $\Exc_{\max}$ and $\diam_{\max}$ be the maximum excess and diameter of the connected components in the \ErdosRenyi random graph $\mathbb{G}(n,1/n)$, respectively. Then, with $\bP$-probability at least $1 - n^{-1}$, we have
    \begin{equation*}
    \diam(\cT) \leq 2 \max_{u \in V} d_\cT(u, \mathcal{L})
    + 2 \max_{2 \leq i \leq k}  d_\cT(\cC_i, \cC_1) + 12( \Exc_{\max} + 2) \diam_{\max}.
\end{equation*}
\end{lemma}

\begin{proof}
    Any path inside $\cT$ between any two vertices $u,v$ can be decomposed into the union of
\begin{enumerate}
    \item a path from $u$ to $\mathcal{L}$;
    \item a path inside some component of $\mathcal{L}$ and a path from said component to $\mathcal{C}_1$;
    \item a path inside $\mathcal{C}_1$;
    \item a path from $\mathcal{C}_1$ to some component of $\mathcal{L}$ and a path inside said component;
    \item a path from $\mathcal{L}$ to $v$. 
\end{enumerate}
This argument gives us that for the UST on $G'$, the graph with the components contracted as in Section~\ref{SS:contract}, we have
\begin{align*}
    \diam(\cT) &\leq  2 \max_{u \in V} d_\cT(u, \mathcal{L})
    + 2 \Big( \max_{2 \leq i \leq k} \big( d_\cT(\cC_i, \cC_1) + \max_{x,y \in \cC_i} d_\cT(x,y) \big) \Big) + \max_{x,y \in \cC_1} d_\cT(x,y) \\
    &\leq  2 \max_{u \in V} d_\cT(u, \mathcal{L})
    + 2 \max_{2 \leq i \leq k}  d_\cT(\cC_i, \cC_1) + 3 \max_{\cC_i} \diam( T_{\textrm{d-max}} (\cC_i) ).
\end{align*}
We note that by Lemma~\ref{L:UST_path_inside}, considering this contraction gives a $\bP$-probability error term of at most $n^{-1}$. Now applying Lemma~\ref{L:excess_diam} gives that
\begin{equation*}
    \diam( T_{\textrm{d-max}} (\cC_i) ) \leq 2(\Exc(\cC_i) + 2) \diam(\cC_i) + \Exc(\cC_i) + 1,
\end{equation*}
from which the lemma follows after some simple algebraic manipulations.
\end{proof}

We now proceed to finish the proof of Theorem~\ref{T:diam_extreme}, where the main contribution of the bound in the diameter will come from the term $d_\cT(\cC_i, \cC_1)$ appearing in Lemma~\ref{L:cT_exc_dia_C1}.

\begin{proof}[Proof of the upper bound in Theorem~\ref{T:diam_extreme}]
Fix $\varepsilon > 0$.  We implicitly work on the graph $G'$ where each component is contracted arbitrarily as in \eqref{eq:Spatial_Markov_Tk}. Consider the following two (good) events that depend only on the randomness of the \ErdosRenyi random graph. For some $B, L> 0$ depending on $\varepsilon$, define:
\begin{align}
        \mathcal{G}_1 &\coloneqq \big\{ \max_{\cC_i} \diam(\cC_i) \leq B n^{1/3} \big\}, \\
        \mathcal{G}_2 &\coloneqq \big\{ \max_{\cC_i} \Exc(\cC_i) \leq L \big\}. \label{eq:max_excess_bounded}
\end{align}
On $\mathcal{G}_1$ and $\mathcal{G}_2$, Lemma~\ref{L:cT_exc_dia_C1} gives that
\begin{equation}
     \diam(\cT) \leq 2 \max_{u \in V} d_\cT(u, \mathcal{L}) \\
    + 2 \max_{2 \leq i \leq k} d_\cT(C_i, \cC_1) + 12 (L+2) B n^{1/3} \label{eq:core_upper_bound}
\end{equation}
with $\bP$-probability at least $1 - n^{-1}$. Formally, we should be multiplying the previous quantities by the indicators of these good events but, for the sake of readability, we refrain from doing so here as well as in the rest of the proof. As a result of Theorem~\ref{T:Nachmias_diam}, the event $\mathcal{G}_1$ occurs with probability at least $1 - \varepsilon/6$ if $B$ is chosen large enough. Furthermore, Theorem~1 in \cite{LPW94} states that the maximum excess of any component is bounded in probability, so that $\mathcal{G}_2$ occurs with probability at least $1 - \varepsilon/6$ if $L$ is large enough. It remains to bound the first two terms in \eqref{eq:core_upper_bound}.

Let 
\begin{equation} \label{eq:upper_bound_goal}
    U \coloneqq A n^{1/3} \log n \log \log n
\end{equation}
for some large $A > 0$ to be chosen later. A simple union bound directly yields that 
\begin{align}
    \bP_\cT \Big(\max_{u \in V} d_\cT (u, \mathcal{L}) > U \Big) &\leq \sum_{u \in V} \bP_\cT( d_\cT(u, \mathcal{L}) > U), \label{eq:upper_ineq_union_bound_1}\\
     \bP_\cT \Big(\max_{2 \leq i \leq k} d_\cT(\cC_i, \cC_1) > U \Big) &\leq \sum_{i=2}^k \bP_\cT (d_\cT(\cC_i, \cC_1) > U). \label{eq:upper_ineq_union_bound_2}
\end{align}
To bound from above the distance between $u$ (resp.\ $\cC_i$) and $\mathcal{L}$ (resp.\ $\cC_1$) we can use the Aldous-Broder algorithm and run a random walk started at the vertex in $G'$ corresponding to $u$ (resp.\ $\cC_i$) and stopped when it hits the vertices corresponding to the components in the definition of $\mathcal{L}$ (resp.\ the vertex corresponding to $\cC_1$). As our choice of starting vertex is not uniform, we start the random walk at index $0$ (instead of at index equal to $1$) with $X_0$ equal to the component $\cC(u)$ of $u$ (resp.\ $\cC_i$). We then let $(X_i)_{i \geq 1}$ be the vertices visited by this random walk after the first step, which by Lemma~\ref{L:couple_X_S}, we may couple to components $(S_i)_{i \geq 1}$ in $G$ obtained by independently sampling components proportional to their size under the measure $\bP_S(\cdot)$ as in Section~\ref{S:size_biased_GNP}. Define $t^u_{\mathcal{L}}$ (resp.\ $t^i_{\cC_1}$) as the first hitting time of the set corresponding to the components in $\mathcal{L}$ (resp.\ $\cC_1$). 
Our construction then gives that
\begin{equation*}
     d_\cT (u, \mathcal{L}) \leq t^u_{\mathcal{L}} + \diam \big( T_{\textrm{d-max}}(\cC(u)) \big) + \sum_{i=1}^{t^u_{\mathcal{L}}-1} \diam\big(T_{\textrm{d-max}}(S_i)\big), \\
\end{equation*}
where $S_i$ is the component corresponding to $X_i$. Similarly, we also obtain
\begin{equation*}
     d_\cT (\cC_i, \cC_1) \leq t^i_{\cC_1} + \diam \big( T_{\textrm{d-max}}(\cC_i) \big)+ \sum_{i=1}^{t^i_{\cC_1}-1}  \diam\big(T_{\textrm{d-max}}(S_i)\big).
\end{equation*}
By Lemma~\ref{AP_T:C_1_size} and Lemma~\ref{L:union_of_largest}, we have that
\begin{equation*}
    \frac{n}{|\mathcal{L}|} \leq \frac{n^{1/3}}{\log n} \quad \text{ and }   \quad  \frac{n}{|\cC_1|} \leq  C n^{1/3}
\end{equation*}
with probability at least $1 - \varepsilon/6$ provided that $C = C(\varepsilon) > 0$ is large enough. In particular, we deduce that the collections $(t^u_{\mathcal{L}} - 1)_{u \in V}$ and $(t^i_{\cC_1} - 1)_{2 \leq i \leq k}$ are stochastically bounded by two families of independent geometric random variables with means $n^{1/3}/\log n$ and $C n^{1/3}$, respectively. As one can easily verify by examining the cumulative distribution function of a geometric random variable, this implies that for some large constant $r = r(C(\varepsilon))> 0$ (independent of $n$) 
\begin{align*}
    \bP_\cT \big(  t^u_{\mathcal{L}}   > r n^{1/3} \big) &= o(n^{-1}), \\
    \bP_\cT \big( t^i_{\cC_1} > r n^{1/3} \log \log n \big) &= o \big( (\log n)^{-3} \big) = o(k^{-1}).
\end{align*}
In view of \eqref{eq:upper_ineq_union_bound_1} and \eqref{eq:upper_ineq_union_bound_2} it therefore suffices to upper bound
\begin{equation}
    \bP_S\Big ( \sum_{i=1}^{r n^{1/3} \log \log n}  \diam\big(T_{\textrm{d-max}}(S_i)\big) > U - r n^{1/3} \log \log n - 3(L+2) Bn^{1/3} \Big), \label{eq:sum_components_diam}
\end{equation}
where we bounded the maximum diameter of the trees $T_{\textrm{d-max}}(\cdot)$ using Lemma~\ref{L:excess_diam}. As $U \gg n^{1/3} \log \log n$, we may also ignore the terms in the right hand side of \eqref{eq:sum_components_diam} by increasing the value of $A$ in the definition of $U$ in \eqref{eq:upper_bound_goal}.

In the following, we consider two more (good) events, depending only on the randomness of the \ErdosRenyi random graph, which will imply that \eqref{eq:sum_components_diam} becomes small, provided that $U$ is large enough. For some $D = D(\varepsilon)> 0$, consider the events:
\begin{align}
        \mathcal{G}_3 &\coloneqq  \big\{ \bE_S\Big[ \diam(S_1) \wedge ( B n^{1/3})  \Big] \leq D \log n  \big\}, \label{eq:1st_moment_small}\\
        \mathcal{G}_4 &\coloneqq  \big\{ \bE_S\Big[ \Big( \diam(S_1) \wedge ( B n^{1/3}) \Big)^2 \Big] \leq D  n^{1/3}  \big\}, \label{eq:2nd_moment_small}
        %
\end{align}
where we recall that $B = B(\varepsilon) > 0$ is some large constant which comes from the event $\mathcal{G}_1$. Corollary~\ref{Cor:gen_moments} together with Markov's inequality gives that events $\mathcal{G}_3$ and $\mathcal{G}_4$ hold with probability at least $1 - 2 \varepsilon/6$ provided that $D = D(B(\varepsilon)) >0$ is large enough.

Let $R := r n^{1/3} \log \log n$ and consider the sum
\begin{equation*}
    \sum_{i=1}^{R}  \diam\big(T_{\textrm{d-max}}(S_i)\big)
\end{equation*}
from \eqref{eq:sum_components_diam}, assuming that the events $\mathcal{G}_i$, $i=1,2,3,4$ hold.  As the excess is not too large (under the event $\mathcal{G}_2$), applying Lemma~\ref{L:excess_diam} shows that this term can be upper bounded by
\begin{equation*}
    2(L+2) \sum_{i=1}^{R}  \diam\big(S_i\big) + (L+1)R .
\end{equation*}
Since the diameter of each component is not larger than $Bn^{1/3}$ (under the event $\mathcal{G}_1$), we may also equivalently bound this term by
\begin{equation} \label{eq:Exc_diam_upper_bound}
     2(L+2) \sum_{i=1}^{R}  \big( \diam (S_i) \wedge (Bn^{1/3}) \big) + (L+1) R.
\end{equation}
Next we define
\begin{equation*}
    \mathcal{X} :=  \sum_{i=1}^{R} (\diam(S_i) \wedge ( B n^{1/3})).
\end{equation*}
and claim that this random variable concentrates rather well. Indeed, since each term involved in the sum is bounded by $Bn^{1/3}$, Bernstein's inequality (see \cite[Theorem~2.8.4]{Ver18}) shows that there exists a constant $c >0$ such for any $t\geq 0$ 
\begin{equation*}
    \bP_S \big( \mathcal{X} - \bE_S[\mathcal{X}] \geq t  \big)  \leq \exp\Big( -c \frac{t^2}{R \cdot \bE_S \big[( \diam(S_1) \wedge (B n^{1/3})^2) \big] + t B n^{1/3}}\Big)
\end{equation*}
Hence, on the events $\mathcal{G}_3$ and $\mathcal{G}_4$, applying the above inequality with $t = D R \log n$ we get 
\begin{align}
     \bP_S \big( \mathcal{X} &\geq  2  D R \log n \big)  \nonumber \\
     & \leq  \bP_S \big( \mathcal{X} - \bE_S[\mathcal{X}] \geq    D R \log n \big) \nonumber \\
    & \leq \exp \Big( -c \frac{( D R \log n )^2}{ D R n^{1/3} +   (D R \log n) \cdot B n^{1/3} }  \Big). \label{eq:Bernstein_bound}
\end{align}
Notice that the exponent in \eqref{eq:Bernstein_bound} is at least $c' (\log n)(\log \log n)$ for some constant $c ' > 0$, and therefore 
\begin{equation} \label{eq:poly_upper_sum_bound}
     \bP_S( \mathcal{X} \geq  2  D R \log n) = O(n^{-2}).
\end{equation}
We conclude the proof by summarizing and combining the bounds that we have obtained so far. Recall that $R = r n^{1/3} \log \log n$. Now let $A = 6D(L+2) r $ so that $U$ from \eqref{eq:upper_bound_goal} satisfies
\begin{equation*}
    \frac{U - (L+1)R}{2(L+2)} = \frac{6D(L+2) R \log n - (L+1)R}{2(L+2)}> 2DR \log n
\end{equation*}
for $n$ large enough. Under the events $\mathcal{G}_i$, $i=1,2,3,4$, 
the bounds from \eqref{eq:sum_components_diam}, \eqref{eq:Exc_diam_upper_bound} and \eqref{eq:poly_upper_sum_bound} yield
\begin{align*}
    \bP_\cT \big( d_\cT(u, \mathcal{L}) > U \big) &\leq \bP_S \Big( \sum_{i=1}^{R} \diam(S_i) \wedge ( B n^{1/3}) > \frac{U - (L+1)R }{2(L+2)} \Big) +o(n^{-1})\\
    &= o(n^{-1}).
\end{align*}
Similarly, the distances $d_{\cT}(\cC_i, \cC_1)$ are larger than $U$ with $\bP$-probability at most $o( k^{-1})$. The inequality in \eqref{eq:core_upper_bound}, and the union bounds in \eqref{eq:upper_ineq_union_bound_1} and \eqref{eq:upper_ineq_union_bound_2}, then give that with $\bP$-probability at least $1 - o(1)$
\begin{equation*}
    \diam(\cT) \leq 4 U + 12(L+2) Bn^{1/3} \leq 5 U
\end{equation*}
for large enough $n$. The proof is completed by noting that the intersection of the events $\mathcal{G}_i$, $i=1,2,3,4$, as well as the event $\{ \mathcal{L} \geq n^{1/3} \log n \}$ and other high probability events, 
occur with probability at least $1 - 5\varepsilon/6 - o(1)$.
\end{proof}

\section{Small and intermediate \texorpdfstring{$\gamma$}{gamma}} \label{S:small_gamma}

In this final section, we show that when $\gamma < 0$ the diameter of the spanning tree is of the same order as that of the unweighted UST. Furthermore, we state a conjecture about the diameter in the case of small $\gamma \geq 0$. Namely, we expect the diameter to be a power law with an exponent decreasing as $\gamma$ increases. 

However, before treating the $\gamma < 0$ case, we give a brief summary of the tools developed for the UST in \cite{MNS21}, which form the basis of several arguments in \cite{MSS23, MSS24}. Theorem~1.1 of \cite{MNS21} gives 3 conditions on the graph and the corresponding random walk that, when satisfied, guarantee that the unweighted UST has a diameter of order $\sqrt{|V|}$ with probability $1-\varepsilon$ (the exact constants of the bound depend on $\varepsilon$). In \cite[Theorem~2.3]{MSS23}, the authors slightly generalized this result to weighted graphs, and applied it to some instances of a random spanning tree in random environment. More specifically, to prove that the diameter is of order $\sqrt{|V|}$, it suffices to show that the balanced, mixing, and escaping conditions (see \cite[(2.9)--(2.11)]{MSS23} or \cite[Section 1.1]{MNS21}) are verified for the (randomly) weighted graph $(G, \weight)$ with some constants independent of $n$, and with large enough probability. 

The balanced condition states that for some $D > 0$
    \begin{equation}
        \frac{\max_{u \in V} \pi(u)}{\min_{v \in V} \pi(v)} = \frac{\max_{u \in V} \sum_{x \neq u} \weight(u,x)}{\min_{v \in V} \sum_{v \neq u} \weight(v,x)} \leq D, \label{eq:bal}
    \end{equation}
    where $\pi(v) = \sum_{x \neq u} \weight(u,x) / 2\sum_{e \in E} \weight(e)$ is the stationary distribution of the random walk. For the mixing and escaping conditions, as demonstrated in \cite[Section 3.2]{MSS24}, it suffices to show that for some fixed $M > 0$ we have the following bound on the \textit{bottleneck ratio}:
    \begin{equation}
        \Phi_{(G,\weight)} = \min\limits_{0 < \pi(S) \leq 1/2 }\Phi(S) := \frac{\sum_{x \in S, y \in S^c} \pi(x) p(x,y)}{\pi(S)} = \frac{\sum_{e \in E(S,S^c)} \weight(e)}{2\sum_{e \in E(S,V)} \weight(e)} \geq M, \label{eq:bottle_neck}
    \end{equation}
    where $E(A,B)$ are the set of edges between $A$ and $B$, and $\pi(S) = \sum_{v \in S} \pi(v)$. We refer to \cite[Section 2]{MSS24} for more details about the bottleneck ratio. We shall apply this method to prove the following proposition, omitting some details in the argument.

\begin{proposition} \label{P:negative_gamma}
    Let $\varepsilon > 0$. There exists a constant $A = A(\varepsilon) > 0$ such that for any $\gamma < 0$ there is an $n_0 = n_0(\gamma, \varepsilon)$ with
    \begin{equation*}
        \E\Big[ \bP_\cT \big( A^{-1} \sqrt{n} \leq \diam(\cT) \leq A \sqrt{n} \big) \Big] \geq 1 - \varepsilon
    \end{equation*}
    for all $n \geq n_0$.
\end{proposition}

\begin{proof}
    Let $\varepsilon > 0$. The degree of a vertex in $\mathbb{G}(n,1/n)$ is distributed as a $\text{Bin}(n-1,1/n)$ random variable, so that combining a union bound together with a Chernoff-type estimate gives that the maximum degree in $\mathbb{G}(n,1/n)$ is bounded by, say, $10 \log n$ with probability $1 - O(n^{-2})$. The inequality in \eqref{eq:bal} is then, for $n$ large enough, satisfied with
    \begin{equation*}
        D = \frac{n-1 + 10 n^{1+\gamma} \log n}{n-1} \leq 2.
    \end{equation*}
    Consider $S$ with $0 < \pi(S) \leq 1/2$. Using that $\weight(e) \geq 1$ and that $x/(x+a)$ is a non-decreasing function in $x$ whenever $a \geq 0$, we obtain
    \begin{align}
        \Phi(S) &= \frac{\sum_{e \in E(S,S^c)} \weight(e)}{2\big( \sum_{e \in E(S,S^c)} \weight(e) + \sum_{e \in E(S,S)} \weight(e) \big)} \geq \frac{|E(S,S^c)|}{2 \big( |E(S,S^c)| + \sum_{e \in E(S,S)} \weight(e) \big)} \nonumber \\
        &\geq \frac{|E(S,S^c)|}{2 \big( |E(S,S^c)| + |E(S,S)| + n^{1+\gamma} \sum_{e \in E(S,S)} \mathbbm{1}_{\{ \weight(e) = n^{1+\gamma}\}} \big)} \nonumber \\
        &\geq\frac{|S|(n-|S|)}{2 \big( |S|(n-|S|) + |S|^2/2 + n^{1+\gamma} |E(S,S) \cap E(\mathbb{G}(n,1/n))|}. \label{eq:bound_bottle}
    \end{align}
    To further bound the bottleneck ratio of $S$, it therefore suffices to upper bound the number of edges in the subgraph (which we shall call $G|_S$) of $\mathbb{G}(n, 1/n)$ induced by the vertex set $S$. To this end, recall that $\Exc(H) = |E(H)| - |V(H)|$ was defined to be the \textit{excess} of a graph $H$, and let $S_1, \ldots, S_k$ be the connected components in $G|_S$. Then
    \begin{align*}
        |E(S,S) \cap E(\mathbb{G}(n,1/n))| &= |E(G|_S)| = \sum_{i=1}^k |E(S_i, S_i)| \\
        &= \sum_{i=1}^k \Exc(S_i) + |V(S_i)| \\
        &\leq \max\limits_{j=1,\ldots,k} ( \Exc(S_j) \vee 2) \sum_{i=1}^k |V(S_i)| \\
        &= \max\limits_{j=1,\ldots,k} ( \Exc(S_j) \vee 2) |S|,
    \end{align*}
    where we take the maximum of the excess with $2$ to ensure that the inequality $a + b \leq ab$ holds (the excess may be equal to $-1,0$ or $1$). Now note that 
    \begin{equation*}
        \max\limits_{j=1,\ldots,k} \Exc(S_j) \leq \max\limits_{i} \Exc(\cC_i),
    \end{equation*}
    as any (non-empty) induced subgraph of a connected graph cannot have a larger excess (adding a vertex to the induced subgraph would add at least one edge as well) than the original graph. By Theorem~1 of \cite{LPW94} (see also equation \eqref{eq:max_excess_bounded}), the maximum excess of a critical \ErdosRenyi random graph is bounded in probability, i.e., there exists a $B = B(\varepsilon)$ such that, with probability at least $1 - \varepsilon/3$, the maximum excess is bounded by $B$. From \eqref{eq:bound_bottle} we then obtain that
    \begin{equation}
        \Phi(S) \geq \frac{|S|(n-|S|)}{2 \big( |S|(n-|S|) + |S|^2/2 + B n^{1+\gamma} |S|\big)} \label{eq:bottle_exces}
    \end{equation}

    To finish the proof, it suffices to control the size of $S$. Since $\pi(S^c) \geq 1/2$ and
    \begin{equation*}
         |S^c| \geq \pi(S^c) \big(\max_{v \in V}\pi(v)\big)^{-1} \geq \frac{1}{2} \cdot \frac{2 |E(K_n)|}{n-1 + 10 n^{1+\gamma} \log n},
    \end{equation*}
    where we used the argument above about the maximum degree in $\mathbb{G}(n,p)$ that holds with large enough probability,  we have 
    \begin{equation*}
        |S| = n - |S^c| \leq \big( 1 - \frac{1}{2 + 20 n^\gamma \log n} \big) (n-1) \leq \frac{2n}{3}
    \end{equation*}
    whenever $n$ is large enough. With \eqref{eq:bottle_exces} this gives that
    \begin{align*}
        \Phi_{(G,\weight)} = \min\limits_{0 < \pi(S) \leq 1/2} \Phi(S) &\geq \min\limits_{0 < |S| \leq 2n/3} \frac{n - |S|}{2 \big( n-|S|/2 + B n^{1+\gamma}  \big)} \\
        &\geq \frac{1}{6\big( 2/3 + B n^{\gamma}\big)},
    \end{align*}
    where the minimum is achieved by a set with maximum size equal to $2n/3$. When $\gamma < 0$, the bottleneck ratio is thus bounded below by a constant as required in \eqref{eq:bottle_neck}. To conclude, there exists an event $\mathcal{A}$ 
    that occurs with probability at least $1 - \varepsilon/2$ such that the conditions required of Theorem~2.3 in \cite{MSS23} (with the bottleneck ratio argument stated before the proposition in mind) are satisfied. Applying Theorem~2.3 in \cite{MSS23} on the event $\mathcal{A}$ with $\varepsilon' = \varepsilon/2$ completes the proof.
\end{proof}

Theorem~\ref{T:diam_extreme} and Proposition~\ref{P:negative_gamma} cover the case $\gamma \not\in [0,5)$ of the weights introduced in \eqref{eq:weight_dist}. It therefore remains open what occurs in the intermediate regime of $\gamma \in [0,5)$, where we expect an interesting behavior to emerge when $\gamma$ is close to $0$. We conjecture the following. 

\begin{conjecture} \label{C:small_gamma}
    There exists constants $c_1, c_2, \gamma^*$ such that for any $\epsilon > 0$, $\gamma \leq \gamma^*$ and $n \geq n_0(\varepsilon)$, we have
    \begin{align}
        \E\Big[ \bP_\cT \big( n^{1/2 - c_1 \gamma} \leq \diam(\cT) \leq n^{1/2 - c_2 \gamma} \big) \Big] \geq 1 - \varepsilon.
    \end{align}
\end{conjecture}

\begin{remark}
    The proof of the Proposition~\ref{P:negative_gamma} shows that for $\gamma \geq 0$ there exists a constant $C > 0$ such that we have the bounds
 \begin{equation*}
     D \leq C n^{\gamma} \log n \qquad \text{and} \qquad  \Phi_{(G,\weight)} \geq C n^{-2 \gamma} (\log n)^{-1}.
 \end{equation*}
 In particular, applying Theorem~2.3 of \cite{MSS23} shows that there exists a constant $c > 0$ such that for small $\gamma \geq 0$ and any $\varepsilon > 0$, one has with probability at least $1 - \varepsilon$ that
    \begin{equation*}
        \diam(\cT) \geq A(\varepsilon) n^{1/2 - c \gamma },
    \end{equation*}
    for some constant $A(\varepsilon)$ depending only on $\varepsilon$. The difficulties in proving Conjecture~\ref{C:small_gamma} therefore reduces to showing a suitable upper bound for the diameter.
\end{remark}

\bibliographystyle{plain}
\bibliography{RSTRE}

\end{document}